\documentclass[11pt]{amsart}
\usepackage{fullpage}
\usepackage{graphicx}
\usepackage{amsfonts}
\usepackage{amsthm}
\usepackage{amsmath}
\usepackage{amssymb}
\usepackage{hyperref}
\usepackage[arrow,matrix,curve,color]{xy}
\usepackage{pb-diagram,pb-xy}
\usepackage{hyperref}   
\usepackage{enumitem}
\usepackage[normalem]{ulem}
\usepackage{color}
\usepackage[normalem]{ulem}
\usepackage{tikz}
\usetikzlibrary{snakes}
\usepackage{xspace}
\definecolor{light-blue}{rgb}{0.8,0.85,1}
\definecolor{light-red}{rgb}{1,.4,.4}
\definecolor{purp}{rgb}{.7,.3,1}
\definecolor{yel}{rgb}{1,1,.5}
\definecolor{cy}{rgb}{0,1,1}
\usepackage{colortbl}
\usepackage{multirow}
\usepackage{array}

\theoremstyle{plain}
\newtheorem{theorem}{Theorem}[section]
\newtheorem{corollary}[theorem]{Corollary}
\newtheorem{lemma}[theorem]{Lemma}

\theoremstyle{definition}
\newtheorem{remark}[theorem]{Remark}
\newtheorem{definition}[theorem]{Definition}

\newtheorem{conjecture}[theorem]{Conjecture}
\newtheorem{question}[theorem]{Question}

\newcommand{\p}{\partial}

\newcommand{\spin}{\mathsf{spin}}

\newcommand{\psc}{positive scalar curvature}
\newcommand{\co}{\colon\,}

\newcommand{\bR}{\mathbb R}
\newcommand{\bC}{\mathbb C}

\newcommand{\bH}{\mathbb H}

\newcommand{\bZ}{\mathbb Z}

\newcommand{\bP}{\mathbb P}

\newcommand{\cK}{\mathcal K}

\newcommand{\cR}{\mathcal R}

\newcommand{\pt}{\text{\textup{pt}}}
\newcommand{\lp}{\textup{(}}
\newcommand{\rp}{\textup{)}}

\newcommand{\Ca}{$C^*$-algebra}
\newcommand{\Cl}{\mathcal{C}\ell}

\newcommand{\per}{\operatorname{per}}

\newcommand{\Dbl}{\operatorname{Dbl}}

\newcommand{\homdim}{\operatorname{hom\,\,dim}}

\newcommand{\Dirac}{\partial\!\!\!/}

\title[PSC on Doubles]{Positive scalar
  curvature on manifolds with boundary and their doubles}
\author{Jonathan  Rosenberg}
\address{Department of Mathematics\\ University of
  Maryland\\ College Park, MD 20742, USA} \email[Jonathan
  Rosenberg]{jmr@umd.edu}
\urladdr{http://www2.math.umd.edu/\raisebox{-3pt}{~}jmr/}
\author{Shmuel Weinberger}\thanks{Partially supported by NSF grant 1811071.}
\address{Department of Mathematics\\ University of
  Chicago\\ Chicago, IL 60637, USA}
\urladdr{http://math.uchicago.edu/\raisebox{-3pt}{~}shmuel/}
  
\begin{document}
\begin{abstract}
  This paper is about positive scalar curvature on a compact
  manifold $X$ with non-empty boundary $\partial X$.  In some cases,
  we completely answer the question of when $X$ has a positive
  scalar curvature metric which is a product metric near $\partial X$,
  or when $X$ has a positive scalar curvature metric with positive
  mean curvature on the boundary, and more generally,
  we study the relationship between boundary conditions on
  $\partial X$ for {\psc} metrics on $X$ and the positive scalar curvature
  problem for the double $M=\Dbl(X,\partial X)$.
\end{abstract}
\dedicatory{Dedicated to Blaine Lawson on his 80th birthday, with
  appreciation and admiration}

\maketitle

\vspace*{-10mm}

\section{Introduction}
\label{sec:intro}

This paper is motivated by two important theorems of Blaine Lawson
(one with Misha Gromov and one with Marie-Louise Michelsohn)
relating curvature properties of a compact manifold with
non-empty boundary to mean curvature of the boundary:
\begin{theorem}[{\cite[Theorem 5.7]{MR569070}}]
  \label{thm:GLdoubling}
  Let $X$ be a compact manifold with boundary, of dimension $n$, and
  let $M=\Dbl(X,\partial X)$ denote the double of $X$ along the boundary
  $\p X$.  If $X$ admits a metric of {\psc} with positive mean
  curvature $H>0$ along the boundary $\p X$
  {\lp}with respect to the outward-pointing normal{\rp}, then $M$ admits a
  metric of {\psc}.
\end{theorem}
\begin{theorem}[{\cite[Theorem 1.1]{MR759265}}]
  \label{thm:MichelsohnLawson} If
  $X$ is a compact manifold with non-empty boundary, of dimension $n\ne 4$,
  then $X$ admits a Riemannian metric with positive sectional curvature
  and positive mean  curvature along the boundary $\p X$
  {\lp}with respect to the outward-pointing normal{\rp}
  if and only if $\pi_1(X,\p X)=*$. Furthermore
  \textup{({\cite[Theorem 1.2]{MR759265}})}, if $X$ is parallelizable,
  then one can replace positive sectional curvature in this result
  by {\bfseries constant} positive sectional curvature.
\end{theorem}
Theorem \ref{thm:MichelsohnLawson}, in turn, is a consequence of an
intermediate result which we will also need:
\begin{theorem}[{\cite[Theorem 3.1]{MR759265}}]
  \label{thm:handleattach} If
  $M$ is a {\lp}normally oriented{\rp} hypersurface of positive mean
  curvature in a Riemannian manifold $\Omega$ of dimension $n$,
  and if $M'\subset\Omega$ is
  obtained from $M$ by attaching a $p$-handle to the positive side
  {\lp}the side of increasing area{\rp} of
  $M$, then if $n-p\ge 2$, one can arrange
  {\lp}without changing the metric on $\Omega${\rp} for $M'$ also to have
  positive mean curvature, to be as close as one wants to $M$,
  and to agree with $M$ away from a small neighborhood of the
  $S^{p-1}\hookrightarrow M$ where the handle is attached.
\end{theorem}

\begin{remark}
\label{rem:signconv}
One should note that there are two different (and conflicting!)
sign conventions in use
for mean curvature, so that the condition $H>0$ along the boundary $\p X$
with respect to the outward-pointing normal in Theorems \ref{thm:GLdoubling}
and \ref{thm:MichelsohnLawson} is what is called in other papers
(such as \cite{MR793350,BH}) $H>0$ with respect to the
\emph{inward}-pointing normal.
We will stick with the terminology in \cite{MR569070}, which means that
close hypersurfaces parallel to the boundary in the interior have
\emph{smaller} volume than the boundary, and that volume grows
as one moves \emph{outward}.  This condition is sometime called
``strictly mean convex boundary.''
\end{remark}

Other proofs of Theorem \ref{thm:GLdoubling} can be found in
\cite[Theorem 1.1]{MR793350} and in \cite[Corollary 34]{BH}.
These two references make it clear that one can weaken the condition
$H>0$ to $H\ge 0$.  However, as Christian B{\"a}r pointed out to
us, the theorem fails if one replaces $H>0$ by $H<0$. He kindly
provided us with the following simple counterexample.  Let $X$
be $S^2$ with three open disks removed, where each disk
fits within a single hemisphere.  (See Figure \ref{fig:3hole}.)
With the restriction of the
standard metric on $S^2$, $X$ has positive curvature (in fact $K=1$)
and the mean curvature of each boundary circle is \emph{negative}, as
parallel circles slightly inward from the boundary components have
\emph{bigger} length.  But the double of $X$ is a surface of genus $2$,
which cannot have nonnegative scalar curvature (by Gauss-Bonnet).
The example can be jacked up to any higher dimension by crossing with
a torus.

\begin{figure}[hbt]
  \begin{center}
    \includegraphics[width=5cm]{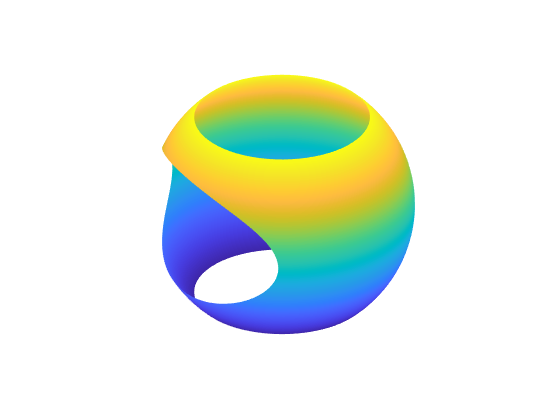}
  \end{center}
  \caption{\label{fig:3hole} A $3$-holed sphere}
\end{figure}

Many years ago, we began studying
whether there might be a sort of converse to Theorem \ref{thm:GLdoubling}.
Recently, Christian B{\"a}r and Bernhard Hanke \cite{BH} have
made a comprehensive study of boundary conditions (mostly involving
mean curvature) for scalar curvature properties of compact manifolds
with non-empty boundary, and this makes it possible to reexamine the
question of a possible converse to Theorem \ref{thm:GLdoubling},
which we have formulated as Conjecture \ref{conj:doubling},
the ``Doubling Conjecture.''
That is the principal subject of this paper.  Our main results on this
question are Theorems \ref{thm:sc}, \ref{thm:scnonspin},
\ref{thm:gencase}, and \ref{thm:gencasenonspin}.

Another closely related problem which we also study
(in Section \ref{sec:obstr}) is the question
of when a compact manifold with boundary admits a {\psc} metric
which is a product metric in a neighborhood of the boundary.
We show that in optimal situations (depending on the fundamental
groups and whether or not things are spin and of high enough dimension)
it is possible to give necessary and sufficient conditions for this to
happen.

Since the proof of Theorem \ref{thm:GLdoubling} in \cite{MR569070} is a bit
sketchy and it seems some of the formulas there are not
completely correct, we have redone the proof in Section
\ref{sec:GL}.

We would like to thank the referee and Rudolf Zeidler for helpful
corrections to the first draft of this paper.

\section{The Relatively $1$-Connected Case}
\label{sec:reloneconn}
Because of Theorem \ref{thm:MichelsohnLawson}, the simplest case to
deal with is the one where $X$ and $\p X$ are connected
and the inclusion $\p X\hookrightarrow X$ induces an isomorphism on
$\pi_1$.  This case should philosophically be viewed as an
analogue of Wall's ``$\pi$-$\pi$ Theorem'' \cite[Theorem 3.3]{MR1687388},
which says that relative surgery problems are unobstructed if
$X$ and $\p X$ are connected
and the inclusion $\p X\hookrightarrow X$ induces an isomorphism on
$\pi_1$. 
\begin{theorem}
\label{thm:reloneconn}
Let $X$ be a connected compact spin manifold with boundary,
of dimension $n\ge 6$,
with connected boundary $\p X$, and such that the inclusion
$\p X\hookrightarrow X$ induces an isomorphism on $\pi_1$.
Let $M=\Dbl(X,\p X)$ denote the double of $X$
along its boundary, which is a closed spin manifold of dimension $n$.
Then the following statements are always true:
\begin{enumerate}
\item $M$ admits a metric of \psc.
\item $\p X$ admits a metric of \psc.
\item $X$ admits a {\psc} metric which is a product metric in a
  collar neighborhood of $\p X$.
\item $X$ admits a {\psc} metric which gives $\p X$ positive mean curvature
  with respect to the interior normal.
\item $X$ admits a {\psc} metric for which $\p X$ is minimal
  {\lp}i.e., has vanishing mean curvature{\rp}.
\item $X$ admits a {\psc} metric which gives $\p X$ positive mean curvature
  with respect to the outward normal.  
\end{enumerate}
\end{theorem}
\begin{proof}
  (6) is the conclusion of Theorem \ref{thm:MichelsohnLawson}, and (1)
  then follows by Theorem \ref{thm:GLdoubling}.  To prove (2), observe
  that if $\pi=\pi_1(\p X)$, then the classifying map
  $c\co \p X\to B\pi$ extends to a classifying map $\bar c$
  for $X$ because of the fundamental group assumption, and so
  $(\p X, c)$ bounds $(X, \bar c)$ and the class of $c\co \p X\to B\pi$
  in $\Omega_{n-1}^\spin(B\pi)$ vanishes.  Then by the Bordism Theorem
  (\cite[Proposition 2.3]{MR866507} or
  Theorems 4.1 and 4.11 in \cite{MR1818778}), (2) holds.  In fact,
  the proof of the Bordism Theorem also yields (3), because
  the bordism over $B\pi$ from $S^{n-1}$ to $\p X$ obtained by punching out a
  disk from $X$ can be decomposed into surgeries of codimension
  $\ge 3$, and then the Surgery Theorem \cite[Theorem A]{MR577131}
  makes it possible to ``push'' a standard {\psc} metric on the $n$-disk
  that is a product metric near the boundary across the bordism to a
  {\psc} metric on $X$ that is a product metric near $\p X$.

  By the construction in Theorem \ref{thm:GLdoubling}, either in
  \cite{MR569070} or in Section \ref{sec:GL}, $M$ has a {\psc}
  metric which is symmetric with respect to reflection across $\p X$.
  So this metric is what B{\"a}r and Hanke call a ``doubling
  metric'' in \cite{BH}.  This metric necessarily has vanishing
  second fundamental form on $\p X$, so (5) holds.  (Alternatively,
  (5) trivially follows from (3).)  By
  \cite[Corollary 34]{BH}, (4) holds as well.  So we have
  shown that all the conditions hold.
\end{proof}

If one keeps the condition that $\p X$ is connected and the
condition that $\pi_1(\p X)\to \pi_1(X)$ is surjective,  but drops
the condition that $\pi_1(\p X)\to \pi_1(X)$ is injective,
then the theorem has to be modified as follows.

\begin{theorem}
\label{thm:reloneconn1}
Let $X$ be a connected compact manifold with boundary,
of dimension $n\ge 6$,
with connected boundary $\p X$, and such that the inclusion
$\p X\hookrightarrow X$ induces a surjection on $\pi_1$.
Let $M=\Dbl(X,\p X)$ denote the double of $X$
along its boundary, which is a closed manifold of dimension $n$.
Then the following statements are always true:
\begin{enumerate}
\item $M$ admits a metric of \psc.
\item $X$ admits a {\psc} metric which gives $\p X$ positive mean curvature
  with respect to the outward normal.  
\item $X$ admits a {\psc} metric for which $\p X$ is minimal
  {\lp}i.e., has vanishing mean curvature{\rp}.
\item $X$ admits a {\psc} metric for which $\p X$ is totally geodesic
  {\lp}i.e., has vanishing second fundamental form{\rp}.
\end{enumerate}
\end{theorem}
\begin{proof}
  Again, (2) follows from Theorem \ref{thm:MichelsohnLawson} and then
  (1) follows from Theorem \ref{thm:GLdoubling}. \cite[Corollary 34]{BH}
  shows that (3) and (4) then follow.
\end{proof}
Note that under the hypotheses of Theorem \ref{thm:reloneconn1}, it is not
necessary true that $\p X$ admits {\psc}, and so in general
$X$ does not have a {\psc} metric which is a product metric in a
neighborhood of the boundary.  A counterexample is given at the
beginning of Section \ref{sec:gencase}.

For the results of Theorem \ref{thm:reloneconn}, the spin restriction is not
really necessary, but the case where $X$ is not spin but has a spin cover
gets messy.  For this reason, it's convenient to make the following
definition.
\begin{definition}
\label{def:totnspin}
If $X$ is a manifold (with or without boundary), we say it is
\emph{totally non-spin} if the second Stiefel-Whitney class $w_2$
of $X$ is non-zero on the image of the Hurewicz map
$\pi_2(X)\to H_2(X,\bZ)$.  This is equivalent to saying that the
universal cover of $X$ does not admit a spin structure.
\end{definition}
The following result is an example of a modification
to the totally non-spin case.
\begin{theorem}
\label{thm:reloneconnnonspin}
Let $X$ be a connected compact oriented manifold with boundary,
of dimension $n\ge 6$, 
with connected boundary $\p X$, such that $X$ and $\p X$
are totally non-spin, and such that the inclusion
$\p X\hookrightarrow X$ induces an isomorphism on $\pi_1$.
Let $M=\Dbl(X,\p X)$ denote the double of $X$
along its boundary, which is a closed oriented manifold of dimension $n$.
Then the following statements are always true:
\begin{enumerate}
\item $M$ admits a metric of \psc.
\item $\p X$ admits a metric of \psc.
\item $X$ admits a {\psc} metric which is a product metric in a
  collar neighborhood of $\p X$.
\item $X$ admits a {\psc} metric which gives $\p X$ positive mean curvature
  with respect to the interior normal.
\item $X$ admits a {\psc} metric for which $\p X$ is minimal
  {\lp}i.e., has vanishing mean curvature{\rp}.
\item $X$ admits a {\psc} metric which gives $\p X$ positive mean curvature
  with respect to the outward normal.  
\end{enumerate}
\end{theorem}
\begin{proof}
  This is exactly like the proof of Theorem \ref{thm:reloneconn}, with
  spin bordism replaced by oriented bordism and with
  \cite[Proposition 2.3]{MR866507} replaced by \cite[Theorem 2.13]{MR866507}.
\end{proof}

\section{The Simply Connected Case}
\label{sec:sc}

Next, we consider the case where $X$ and all the components
of $\p X$ are simply connected, but there can be multiple
boundary components.
\begin{theorem}
\label{thm:sc}
Let $X$ be a simply connected compact spin manifold with non-empty boundary,
with $n=\dim X\ge 6$. Note that $\p X$ can have any number $k$ of boundary
components, $\p_1 X,\cdots, \p_k X$.
Suppose all components of $\p X$ are simply connected.
Let $M=\Dbl(X,\p X)$ denote the double of $X$
along its boundary, which is a closed spin manifold of dimension $n$.
Then the following are equivalent:
\begin{enumerate}
\item $M$ admits a metric of \psc.
\item All components $\p_j X$ of $\p X$ admit metrics of \psc.
\item For each $j=1,\cdots, k$, the $\alpha$-invariant
  $\alpha(\p_j X)\in ko_{n-1}$ vanishes.
\item $X$ admits a {\psc} metric which is a product metric in a
  collar neighborhood of $\p X$.
\item $X$ admits a {\psc} metric for which $\p X$ is minimal
  {\lp}i.e., has vanishing mean curvature{\rp}.
\item $X$ admits a {\psc} metric which gives $\p X$ positive mean curvature
  with respect to the outward normal.  
\end{enumerate}
\end{theorem}
\begin{proof}
  If the number $k$ of boundary components is $1$, then
  the pair $(X,\p X)$ is relatively $1$-connected and all six conditions
  hold by Theorem \ref{thm:reloneconn}.
  Thus for most of the proof we can restrict to the case $k\ge 2$.
  We begin by observing that since $X$ and all components of $\p X$
  are simply connected, Van Kampen's Theorem
  implies that the fundamental group of $M$ is the same
  as for the graph $\Gamma$ given by
  $\xymatrix{\bullet\ar@{-}[r]&\bullet\ar@{-}[r]&
    \bullet}$ if $k=1$,
  $\xymatrix@R=0pt{&\bullet\ar@{-}[dr]&\\
    \bullet\ar@{-}[ur]\ar@{-}[dr]&&\bullet\\
    &\bullet\ar@{-}[ur]&}$ if $k=2$,
  $\xymatrix@R=0pt{
    &\overbrace{\bullet}^k\ar@{-}[dr]&\\
    \bullet\ar@{-}[ur]\ar@{-}[dr]\ar@{-}[r]&
               {\mbox{\huge$\vdots$}}\ar@{-}[r]&\bullet\\
    &\bullet\ar@{-}[ur]&}$
  for larger $k$, and is thus the free group $F_{k-1}$ on $k-1$ generators.
  This is a group for which \cite[Theorem 4.13]{MR1818778} applies,
  and thus $M$ has a Riemannian metric of {\psc} if and only if the
  obstruction class $\alpha_\Gamma(M)$ (the image of the $ko$-fundamental
  class of the spin manifold $M$ under the classifying map
  $c\co M\to BF_{k-1}\simeq \Gamma$) vanishes in
  $ko_n(\Gamma)\cong ko_n \oplus ko_{n-1}^{k-1}$.  The component of
  $\alpha_\Gamma(M)$ in $ko_n$ is just the ordinary $\alpha$-invariant of
  $M$, i.e., the image of the fundamental class $[M]\in ko_n(M)$
  under the ``collapse'' map $M\to \pt$.  This always vanishes since
  $M$, being a double of a compact spin manifold with boundary,
  is always a spin boundary (the boundary of the manifold obtained
  by rounding the corners of $X\times I$).  Hence condition (1)
  holds if and only if the component of $\alpha_\Gamma(M)$
  in $ko_{n-1}^{k-1}$ vanishes.  It is easy to see that this is the
  same as the vanishing of all the $\alpha$-invariants $\alpha(\p_j X)$
  for all the boundary components $\p_1 X,\cdots, \p_k X$. Indeed, the sum of
  these $k$ $\alpha$-invariants in $ko_{n-1}$ is just
  $\alpha(\p X)=0$, since $\p X$ is a spin boundary and
  $\alpha$ is a spin bordism invariant. (In other words,
  the summand $ko_{n-1}^{k-1}$ in $ko_n(\Gamma)$ is better described as the
  set of elements in $ko_{n-1}^k$ that sum to $0$.)  And we can
  compute the piece of $\alpha_\Gamma(M)$ in one summand of $ko_{n-1}$
  by taking the $\alpha$-invariant of the transverse inverse image
  of a point under
  the map $M\to S^1$ obtained by collapsing $k-2$ of the
  loops in $\Gamma\simeq \bigvee_{k-1} S^1$ and composing with
  the classifying map $c$ as in Figure \ref{fig:fig3}.  Thus (1) is
  equivalent to (3). Applying Stolz's
  Theorem \cite{MR1189863} and using the assumption that $n-1\ge 5$,
  we obtain the equivalence of (2) and (3), and thus (1), (2), and (3)
  are all equivalent.
    \begin{figure}[hbt]
      \begin{tikzpicture}
      \node at (-4.3cm, 0cm) {$M$};
      \draw[thick,black] (-2cm, 0cm) ellipse (2cm and 1.05cm);
      \draw[thick,black] (-2cm, .45cm) ellipse (1cm and .3cm);
      \draw[thick,black] (-2cm, -.45cm) ellipse (1cm and .3cm);
      \draw[thick,red]   (-2cm,.9cm) ellipse (.07cm and .15cm);
      \draw[thick,black]  (-2cm,-.9cm) ellipse (.07cm and .15cm);
      \draw[thick,black]  (-2cm,0cm) ellipse (.07cm and .15cm);
      \draw[thick,blue,->] (0.3cm,0cm) -- (2.3cm, 0cm)
      node[above,midway] {$c$};
      \node at (2.5cm, 0cm) {$\bullet$};
      \node at (3.5cm, 0cm) {$\bullet$};
      \node[red] at (3.0cm, 1cm) {$\bullet$};
      \node at (3.0cm, -1cm) {$\bullet$};
      \draw[thick,black] (3.5cm, 0cm) -- (3.0cm, 1cm) --
      (2.5cm, 0cm) --  (3.5cm, 0cm) -- (3.0cm, -1cm) -- (2.5cm, 0cm);
      \node at (3.3cm, .6cm) {$\Gamma$};
      \draw[thick,blue,->] (3.0cm, -.9cm) -- (3.0, -.1cm);
      \draw[thick,blue,->] (3.8cm, 0cm) -- (4.8cm, 0cm);
      \draw[thick,black]  (5.5cm, 0cm) circle (.5cm);
      \node[red] at (5.5cm, .5cm) {$\bullet$};
    \end{tikzpicture}  
  \caption{\label{fig:fig3} Computing the $\alpha_\Gamma$-invariant}
  \end{figure}
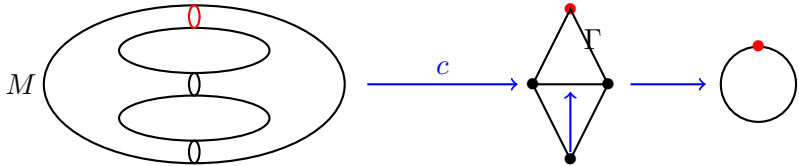

  The equivalence of (2) and (4) is related to Chernysh's
  Theorem (\cite[Theorem 1.1]{MR2213758} and \cite[Theorem 1.1]{MR4236952}
  --- see also \cite[Corollary D]{MR4150632}). Let's explain this
  in more detail, since Chernysh proves that
  the map $\cR^+(X,\p X)\to \cR^+(\p X)$ is a Serre fibration\footnote{Here
    $\cR^+(\p X)$ is the space of {\psc} metrics on $\p X$, and
    $\cR^+(X,\p X)$ is the space of {\psc} metrics on $X$ that restrict
    to a product metric on a collar neighborhood of the boundary.}, but
  not that $\cR^+(X,\p X)\ne \emptyset$.  (In fact, even in nice
  situations such as $X=D^8$ and $\p X=S^7$, the map
  $\cR^+(X,\p X)\to \cR^+(\p X)$ is known not to be surjective.)
  Obviously (4) implies (2). If (2) holds, we need
  to show that for some choice of {\psc} metrics on $\p_j X$, the
  metric on $\p X$ extends to a {\psc} metric on $X$ which is a product
  metric in a collar neighborhood of the boundary. For this we use
  the assumption that $k\ge 2$, and we fix metrics of {\psc} on the
  boundary components $\p_j X$, $1\le j\le k-1$.  Then we can
  view $X$ as a spin cobordism from $\p_1 X\coprod \cdots \coprod \p_{k-1} X$
  to $-\p_k X$.  The Gromov-Lawson surgery theorem
  \cite[Theorem A]{MR577131}, in the variant found in
  \cite{MR2213758,MR4236952,MR4150632}, shows that for \emph{some}
  spin cobordism $W$ between these manifolds, obtained from $X$
  by doing surgeries on embedded $2$-spheres
  with trivial normal bundles, away from the boundary,
  to make the pair $(W, \p_k X)$ $2$-connected
  (the pair $(X, \p_k X)$ is already $1$-connected by hypothesis),
  we can push the {\psc} metric
  on $\p_1 X\coprod \cdots \coprod \p_{k-1} X$ across the cobordism $W$
  to get a {\psc} metric on $W$ restricting to product metrics on
  collar neighborhoods of the boundary components. If $W=X$, we have
  shown that (2) $\Rightarrow$ (4).  In general $X$ and $W$ will not
  be the same, but we can go back from $W$ to $X$ by doing surgery on embedded
  $(n-3)$-spheres away from the boundary, so we can use the surgery
  theorem again to carry the metric of {\psc} over from $W$ to $X$.
  
  Now we need to check
  equivalence of the conditions (1)--(4) with (5) and (6).
  But by \cite[Theorem 33]{BH}, existence of a {\psc} metric
  on $X$ with $H\ge 0$ ($H$ denotes the mean curvature of $\p X$
  with respect to the outward normal) is equivalent to
  existence of a {\psc} metric with $H> 0$. Thus (5) implies
  (6).  It is also trivial that (4) $\Rightarrow$ (5), so
  (1)--(4) implies (5) and (6).  Theorem \ref{thm:GLdoubling}
  shows that (6) implies (1) and thus all the other conditions.
  Thus all six conditions are equivalent.
\end{proof}
\begin{remark}
  If the boundary of $X$ is empty, then $\Dbl(X,\p X)=X\coprod -X$,
  which clearly has a metric of {\psc} if and only if $X$ does.
  So in this case (2) and (3) always hold, (4)--(6) amount to saying
  that $X$ has a metric of {\psc}, and these conditions are equivalent
  to (1).
\end{remark}

Once again, one can easily modify Theorem \ref{thm:sc} to the non-spin
case as follows:
\begin{theorem}
\label{thm:scnonspin}
Let $X$ be a simply connected compact manifold with non-empty boundary, with
$n=\dim X\ge 6$. Note that $\p X$ can have any number $k$ of boundary
components, $\p_1 X,\cdots, \p_k X$.
Suppose all components of $\p X$ are simply connected and that
{\bfseries none} of $X$ and the $\p_j X$ admit a spin structure.
Let $M=\Dbl(X,\p X)$ denote the double of $X$
along its boundary, which is a closed oriented manifold of dimension $n$.
Then the following are true:
\begin{enumerate}
\item $M$ admits a metric of \psc.
\item All components $\p_j X$ of $\p X$ admit metrics of \psc.
\item $X$ admits a {\psc} metric which is a product metric in a
  collar neighborhood of $\p X$.
\item $X$ admits a {\psc} metric for which $\p X$ is minimal
  {\lp}i.e., has vanishing mean curvature{\rp}.
\item $X$ admits a {\psc} metric which gives $\p X$ positive mean curvature
  with respect to the outward normal.  
\end{enumerate}
\end{theorem}
\begin{proof}
  Assertion (2) follows immediately from \cite[Corollary C]{MR577131}. As
  in the proof of Theorem \ref{thm:sc}, $\pi_1(M)\cong F_{k-1}$, a free
  group.  When $k=1$, we can apply Theorem \ref{thm:reloneconnnonspin}.
  So we can assume $k\ge 2$.  To prove (1), we can apply
  \cite[Theorem 4.11]{MR1818778} or \cite[Theorem 1.2]{MR3078256},
  which says that it suffices to show that
  there is a manifold with {\psc} representing the same class as $M$
  in $H_n(BF_{k-1}, \bZ) = H_n(\bigvee_1^{k-1} S^1, \bZ) = 0$
  (since $n>1$). So this is automatic. The proof of the remaining
  conditions is exactly the same as for Theorem \ref{thm:sc}.
\end{proof}  

\section{Obstruction Theory for the Relative Problem}
\label{sec:obstr}
In this section we extend some of the results of Theorem
\ref{thm:reloneconn} and Theorem \ref{thm:sc} to get a general
obstruction theory and a conjectured answer for the following problem:
\begin{question}
\label{q:relobstr}
Suppose that $X$ is a connected compact spin manifold of dimension $n$
with non-empty boundary $\p X$ (which could be disconnected).  Assume
that $\p X$ admits a metric of {\psc}.  (When the fundamental
groups of the components $\p_1 X,\cdots,\p_k X$ of $\p X$ are nice
enough and $n\ge 6$, the Gromov-Lawson-Rosenberg Conjecture holds and
one has a necessary and sufficient condition for this in terms
of $KO$-theoretic ``$\alpha$-invariants'' $\alpha_{\pi_1(\p_j X)} (\p_j X)$.
See Theorem \ref{thm:GLRmain} below.)  When does $X$ admit
a metric of {\psc} which is a product metric in a collar neighborhood
of $\p X$?
\end{question}
To explain our approach to this, we first establish some notation.
The fundamental groupoid $\Lambda$ of $\p X$ is equivalent
to $\coprod_j \Lambda_j$, the disjoint union of the groups
$\Lambda_j=\pi_1(\p_j X)$, where $\p_1 X,\cdots,\p_k X$ are the
components of $\p X$ and we pick a
basepoint in each component.  The classifying space $B\Lambda$ is
homotopy equivalent to $\coprod_j B\Lambda_j$.
There is a classifying map $c^\Lambda\co \p X\to B\Lambda$,
unique up to homotopy equivalence, which is an isomorphism on
fundamental groups on each component.  The spin structure of
$X$ determines spin structures on each $\p_j X$ and thus
$ko$-fundamental classes $[\p_j X]\in ko_{n-1}(\p_j X)$.
By \cite[Theorem 4.11]{MR1818778} (or if you prefer,
\cite[Theorem 1.2]{MR3078256}), the question of whether or
not $\p_j X$ admits a metric of {\psc} only depends on
$c^\Lambda_*([\p_j X])\in ko_{n-1}(B\Lambda_j)$, and so the question of
whether or not $\p X$ admits a metric of {\psc} only depends on
$c^\Lambda_*([\p X])\in ko_{n-1}(B\Lambda)$. Furthermore, the image of
$c^\Lambda_*([\p X])$ under ``periodization'' (inverting the Bott element)
$\per\co ko_{n-1}(B\Lambda)\to KO_{n-1}(B\Lambda)$, followed by the
$KO$-assembly map $A\co KO_{n-1}(B\Lambda) \to KO_{n-1}(C^*_{r,\bR}(\Lambda))$
(or one could use the full $C^*$-algebra here), is an obstruction
to {\psc} on $\p X$.  Thus if for each $j$, periodization and
assembly are injective for $\Lambda_j$, we obtain the
\emph{Gromov-Lawson-Rosenberg Conjecture}:
\begin{theorem}[{\cite[Theorem 4.13]{MR1818778}}]
\label{thm:GLRmain}
Suppose $(X, \p X)$ is a connected compact spin manifold
of dimension $n$ with non-empty boundary $\p X$. With the
above notation, if $\p X$ admits a metric of {\psc},
then $A\circ\per(c^\Lambda_*([\p X]))=0$ in $KO_{n-1}(C^*_{r,\bR}(\Lambda))$.
If $n\ge 6$ and if $A$ and $\per$ are injective for $\Lambda$,
then vanishing
of $c^\Lambda_*([\p X])\in ko_{n-1}(B\Lambda)$ is necessary and sufficient
for $\p X$ to admit a metric of {\psc}.
\end{theorem}

Now in the situation of Theorem \ref{thm:GLRmain}, one also has a
$ko$-fundamental class $[X,\p X]\in ko_n(X, \p X)$,
which maps under the boundary map of the long exact sequence
of the pair $(X,\p X)$ to $[\p X]\in ko_{n-1}(\p X)$.
Let $\Gamma=\pi(X)$ and $\Lambda=\pi(\p X)$ as above.
(Note that we are using the fundamental groupoids to avoid
having to make changes of basepoint.  But there is a natural
map of topological groupoids $\Lambda\to\Gamma$.)
We have the following commuting diagram of $ko$-groups, coming
from the long exact sequences of the
pairs $(X, \p X)$ and $(B\Gamma, B\Lambda)$:
\begin{equation}
\label{eq:koexactseq}
\xymatrix{
  & [X,\p X] \ar@{|->}[r]^\p
  \ar@{}[d] |{\cap\!\!\text{\rule{.3pt}{5pt}}}& [\p X]
  \ar@{}[d] |{\cap\!\!\text{\rule{.3pt}{5pt}}} \\
  ko_n(X) \ar[r] \ar^{c^\Gamma_*}[d]
  & ko_n(X, \p X) \ar[r]^\p \ar^{c^{\Gamma,\Lambda}_*}[d] &
  ko_{n-1}(\p X) \ar^{c^\Lambda_*}[d]\\
  ko_n(B\Gamma) \ar[r] & ko_n(B\Gamma, B\Lambda)\ar[r]^\p &
  \, ko_{n-1}(B\Lambda).}
\end{equation}
One of course gets similar diagrams with $ko$ replaced by
$H$ (ordinary homology), $\Omega^\spin$, and $KO$ (periodic $K$-homology).
If $\p X$ admits a metric of {\psc}, then Question \ref{q:relobstr}
is meaningful.  Here is our first major result on Question \ref{q:relobstr}.
\begin{theorem}
\label{thm:relbordism}
Suppose $(X, \p X)$ is a connected compact spin manifold
of dimension $n\ge 6$ with non-empty boundary $\p X$. Let
$\Gamma=\pi(X)$ and $\Lambda=\pi(\p X)$ be the fundamental groupoids.
As explained above, note that $(X, \p X)$ defines a class
$c^{\Gamma,\Lambda}_*([X, \p X])\in \Omega^\spin_n(B\Gamma,B\Lambda)$.  Suppose
there is another compact spin manifold $(Y, \p Y)$
of the same dimension, 
defining the same class in $\Omega^\spin_n(B\Gamma,B\Lambda)$.  If
$Y$ admits a metric of {\psc} which is a product metric in a
neighborhood of the boundary, then so does $X$.
\end{theorem}
\begin{remark}
\label{rem:pipi}
Note that Theorem \ref{thm:relbordism} includes the assertion
that if $\p X$ is connected and $\Gamma=\Lambda$,
then since in this case the relative groups for $(B\Gamma,B\Lambda)$
vanish for any homology theory, we conclude that $X$ always
has a metric of {\psc} which is a product metric in a neighborhood
of the boundary.  We already know this from Theorem \ref{thm:reloneconn}.
\end{remark}
\begin{proof}
  Note that the condition of the theorem implies allows for $Y$ to
  be disconnected or to have more boundary components than $X$.  But
  the condition implies that there is a spin manifold $W$ with corners,
  of dimension $n+1$, giving a spin bordism from $(Y, \p Y)$ to $(X, \p X)$,
  which restricts on the boundary to a spin bordism of closed
  manifolds from $\p Y$ to $\p X$. Furthermore, this bordism is
  ``over'' $B\Lambda$ on the boundary and $B\Gamma$ on the interior.
  We can number the boundary components of $X$ as $\p_1 X,\cdots, \p_k X$,
  so that $\p Y =\coprod_j \p_j Y$ and $W$ gives a spin bordism
  from $\p_j Y$ to $\p_j X$ over $B\Lambda_j$.  See Figure \ref{fig:bordcorn},
  which illustrates the case $k=2$.

  \begin{figure}[hbt]
    \begin{tikzpicture}
      \fill[red!15!white,] (2cm, 2cm) -- (1.5cm, 2cm)
      arc (90:180:2cm and .7cm) -- (-.5cm, -1.7cm) 
      arc (180:90:2cm and .7cm) -- (2cm, -1cm) -- (2cm, 2cm);
      \fill[semitransparent,red!30!white] (-.5cm, 1.3cm)
      -- (-.5cm, -1.7cm) arc(180:270:2cm and .7cm) -- (1.5cm, 0.6cm)
      arc(-90:-180:2cm and .7cm);
      \draw[very thick,red] (2cm, 2cm) -- (2cm, -1cm);
      \draw[very thick,red] (1.5cm, 0.6cm)  -- (1.5cm, -2.4cm);
      \node[thick,red] at (2cm, 2cm) {$\bullet$};
      \node[thick,red] at (1.5cm, 0.6cm) {$\bullet$};
      \node[thick,red] at (2cm, -1cm) {$\bullet$};
      \node[thick,red] at (1.5cm, -2.4cm) {$\bullet$};      
      \draw[very thick,black] (2cm, 2cm) -- (1.5cm, 2cm)
      arc (90:270:2cm and .7cm);
      \draw[very thick,black] (2cm, -1cm) -- (1.5cm, -1cm)
      arc (90:270:2cm and .7cm);
      \node at (2.45cm, 2cm) {$\p_1 Y$};
      \node at (1.5cm, 2.3cm) {$Y$};
      \node at (1.2cm, 0.3cm) {$\p_2 Y$};
      \node at (2.45cm, -1cm) {$\p_1 X$};
      \node at (1.4cm, -2.7cm) {$\p_2 X$};
      \node at (.5cm, -2.6cm) {$X$};
      \node at (.5cm, -.8cm) {$W$};
    \end{tikzpicture}  
  \caption{\label{fig:bordcorn} The spin bordism $W$ from $(Y, \p Y)$
  to $(X, \p X)$}
  \end{figure}
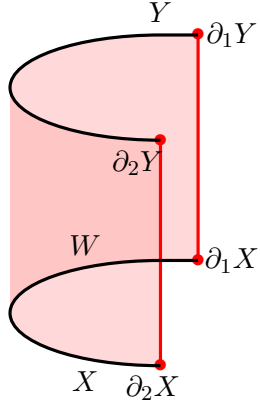

  Now we proceed as in the proof of the Gromov-Lawson bordism theorem
  \cite[Theorem B]{MR577131}, or more exactly of the generalization of
  this to the non-simply connected case \cite[Theorem 4.2]{MR1818778}.
  Start with a metric of {\psc} on $Y$
  that is a product metric in a neighborhood of the boundary.
  We can do surgeries by attaching handles of dimensions $1$ and $2$ to adjust
  $\p Y$ so that it has the same number of components as $\p X$ and so
  that $\p_j Y$ has the same fundamental group $\Lambda_j$ as
  $\p_j X$.  These surgeries are in codimension $3$ or more
  since $\dim \p_j Y\ge 5$, so we can do this preserving {\psc}
  on the successive transformations of $\p_j Y$.  We always extend the
  metric to a collar neighborhood in $W$ so as to be a product metric
  near this component of the boundary.  Similarly, we can then do surgeries
  on the interior of $Y$ and then on the interior of $W$ so that
  these have the same fundamental group $\Gamma$ as $X$.
  Then the bordism from $Y$ to $X$ can be decomposed into a sequence
  of surgeries (over $B\Lambda$ on the bordism from $\p Y$ to $\p X$ and over
  $B\Gamma$ in the interior) in codimension $3$ or more so that
  we can carry the metric across the bordism, preserving the {\psc}
  conditions.  The spin condition is used to know that whenever we want to do
  surgery on an embedded $1$-sphere or $2$-sphere, it has trivial
  normal bundle and thus the surgery is possible.
\end{proof}

The work of Stolz and Jung leading to \cite[Theorem 4.11]{MR1818778}
can be used to give a substantial improvement to Theorem
\ref{thm:relbordism}.  (Note that Jung's work was never published,
but that his results were reproved by F{\"u}hring in
\cite{MR3078256}.)
\begin{theorem}
\label{thm:relko}
Suppose $(X, \p X)$ is a connected compact spin manifold
of dimension $n\ge 6$ with non-empty boundary $\p X$. Let
$\Gamma=\pi(X)$ and $\Lambda=\pi(\p X)$ {\lp}and remember that
$\Lambda$ can have multiple connected components{\rp}.
As explained above, note that $(X, \p X)$ defines a class
$c^{\Gamma,\Lambda}_*([X, \p X])\in ko_n(B\Gamma,B\Lambda)$.  If
this class vanishes in this relative $ko$-homology group, then
$X$ admits a metric of {\psc} which is a product metric in a
neighborhood of the boundary.
\end{theorem}
\begin{proof}
  First of all, note that by Theorem \ref{thm:relbordism}, there is a
  subgroup $\Omega^{\spin,+}_n(B\Gamma,B\Lambda)$ of
  $\Omega^\spin_n(B\Gamma,B\Lambda)$ such that $X$ admits a metric of
  {\psc} which is a product metric in a neighborhood of the boundary
  if and only if $c^{\Gamma,\Lambda}_*([X,\p X])$ lies in this subgroup.
  (This sub\emph{set} is
  a sub\emph{group} since the condition is clearly stable under disjoint union
  and orientation reversal, which represent addition and inversion
  in the spin bordism group.)  So we just need to show that the kernel
  of $\alpha_{(\Gamma,\Lambda)}\co \Omega^\spin_n(B\Gamma,B\Lambda)\to
  ko_n(B\Gamma,B\Lambda)$ lies in $\Omega^{\spin,+}_n(B\Gamma,B\Lambda)$.
  In addition, we can localize and check this separately after localizing
  at $2$ and after inverting $2$.  After localizing at $2$, we can invoke
  \cite[Theorem B]{MR1259520}, which says that after localizing
  at $2$, the kernel of $\Omega^\spin_n(Z,W)\to ko_n(Z,W)$ is, for any
  $Z$ and $W$,
  generated by the image of a transfer map associated to $\bH\bP^2$-bundles.
  What this means geometrically is that there is a generating set for
  the kernel consisting of pairs $(X, \p X)$ (mapping to $(Z,W)$)
  which can be taken (up to
  bordism) to be fiber bundles $\bH\bP^2\to (X, \p X)\to (Y, \p Y)$,
  with structure group the isometry group of $\bH\bP^2$, where $\dim Y = n-8$.
  If $X$ is of this form, just give $Y$ a metric which is a product
  metric near the boundary and lift it to a metric on $X$ which on the
  fibers is the usual metric on $\bH\bP^2$, but rescaled to have very
  large curvature.  This will have {\psc} and still be a product metric
  near the boundary.

  Now we just have to study the kernel of $\alpha$ after inverting $2$.
  For this we can use \cite[Corollary 3.2]{MR3078256}, which says that
  for any space $Z$, the kernel of
  $\alpha\co \Omega_*^\spin(Z)[\frac12]\to ko_*(Z)[\frac12]$
  is generated by manifolds which carry a {\psc} metric.  Almost exactly
  the same argument works in the relative case.
\end{proof}

Incidentally, there is an exact counterpart to Theorems
\ref{thm:relbordism} and \ref{thm:relko} in the totally non-spin case
(Definition \ref{def:totnspin}).
\begin{theorem}
\label{thm:relbordismnonspin}
Suppose $(X, \p X)$ is a connected compact oriented manifold
of dimension $n\ge 6$ with non-empty boundary $\p X$.
Suppose that $X$ and each component of $\p X$ is totally non-spin.  Let
$\Gamma=\pi(X)$ and $\Lambda=\pi(\p X)$.  Suppose
there is another compact oriented manifold $(Y, \p Y)$
of the same dimension, 
defining the same class in $\Omega_n(B\Gamma,B\Lambda)$.  If
$Y$ admits a metric of {\psc} which is a product metric in a
neighborhood of the boundary, then so does $X$.
\end{theorem}  
\begin{proof}
  This is exactly like the proof of Theorem \ref{thm:relbordism} except
  for a few points regarding the totally non-spin condition.  Start with
  an oriented bordism $W$ from $Y$ to $X$, as in Figure \ref{fig:bordcorn}
  (except that it can't be spin since $X$ is not). As in the proof
  of Theorem \ref{thm:relbordism}, start by doing surgeries on the
  boundary so that $\p Y$ has the same number of components as $\p X$,
  and number them so that $\p W$ gives a bordism from $\p_j Y$
  to $\p_j X$ over $\Lambda_j$ and so that the appropriate piece $\p_j W$
  of $\p W$ has fundamental group $\Lambda_j$. Recall that $w_2$
  restricted to the image of the Hurewicz map gives the obstruction
  to triviality of the normal bundle for embedded $2$-spheres.
  Since $w_2(\p_j X)$ is non-zero on the image of the Hurewicz map,
  we consider the commutative diagram
  \[
  \xymatrix{\pi_2(\p_j X)\ar[rr]\ar@{->>}[rd]^{w_2} &
    & \pi_2(\p_j W) \ar[ld]_{w_2}\\ & \bZ/2 &}
  \]
  (where the horizontal map on the top is induced by the inclusion
  $\p_j X\hookrightarrow \p_j W$) and do surgeries on $2$-spheres
  (which have to have trivial normal bundles)
  representing generators of the kernel of $w_2$ on $\pi_2(\p_j W)$,
  in order to make $(\p_j W,\p_j X)$ $2$-connected.  Similarly, we
  can make $(W,X)$ $2$-connected by the same argument.  The rest of the
  proof is as before.
\end{proof}
\begin{theorem}
\label{thm:relhom}
Suppose $(X, \p X)$ is a connected compact oriented manifold
of dimension $n\ge 6$ with non-empty boundary $\p X$.  Assume that
$X$ and each component of $\p X$ is totally non-spin. Let
$\Gamma=\pi(X)$ and $\Lambda=\pi(\p X)$ and 
note that $(X, \p X)$ defines a class
$c^{\Gamma,\Lambda}_*([X, \p X])\in H_n(B\Gamma,B\Lambda; \bZ)$.  If
this class vanishes in this relative homology group, then
$X$ admits a metric of {\psc} which is a product metric in a
neighborhood of the boundary.
\end{theorem}
\begin{proof}
  By Theorem \ref{thm:relbordismnonspin}, there is a subgroup
  $\Omega_n^+(B\Gamma,B\Lambda)$ such that $X$ admits a metric of
  {\psc} which is a product metric in a neighborhood of the boundary
  if and only if the bordism class $c^{\Gamma,\Lambda}_*([X, \p X])$
  lies in this subgroup.  (Again, this is a subgroup and not just a
  subset, for the same reasons as in the proof of Theorem
  \ref{thm:relko}.)
  We just need to show that $\Omega_n^+(B\Gamma,B\Lambda)$ contains the
  kernel of
  $\Omega_n(B\Gamma,B\Lambda)\to H_n(B\Gamma,B\Lambda; \bZ)$.
  As in the spin case, it suffices to do the calculation separately
  first after localizing at $2$ and then after inverting $2$.  If we
  localize at $2$, then it is known\footnote{This is really due to
    Wall \cite{MR120654}, but one can find it more explicitly in
    \cite{MR431142}.}
  that $\mathsf{MSO}$
  becomes an Eilenberg-Mac Lane spectrum, or in other words, the
  Atiyah-Hirzebruch spectral sequence
  $H_p(Z,Y;\Omega_q)\Rightarrow \Omega_{p+q}(Z,Y)$ always collapses
  after localizing at $2$, and
  $\Omega_n(Z,Y)_{(2)}=\bigoplus_{p+q=n} H_p(Z,Y;\Omega_{(2),q})$ for any $Z$
  and $Y$. Apply this with $Z=B\Gamma$ and $Y=B\Lambda$.
  Since Gromov and Lawson showed \cite{MR577131} that
  $\Omega_*$ has a set of additive generators which are closed oriented
  manifolds of {\psc}, the result localized at $2$ clearly follows.
  The result after inverting $2$ follows from \cite{MR3078256}, just as
  in the spin case.
\end{proof}

Now we can give a complete answer to Question \ref{q:relobstr} when the
relevant fundamental groups are nice enough.  First we need to discuss
an obstruction to a positive answer to \ref{q:relobstr}.
This involves {\Ca}ic $K$-theory.
We always work over $\bR$ instead of over $\bC$,
though this makes very little difference in the formal structure of the 
argument. $C^*$-algebras of fundamental \emph{groupoids} may be unfamiliar 
to many readers, but note from the theory of groupoid $C^*$-algebras
\cite{MR873460} that for $X$ a nice
path-connected locally compact space such as a connected manifold,
$C^*(\pi(X))$ is strongly Morita equivalent to the group $C^*$-algebra
$C^*(\pi_1(X,x_0))$, for any choice of a basepoint $x_0$ in $X$.
In fact $C^*(\pi(X))\cong C^*(\pi_1(X,x_0))\otimes \cK$, where $\cK$
is the algebra of compact operators.

The following result is really due to Chang, Weinberger and Yu
\cite[Theorem 2.18]{MR4045309} and to
Schick and Seyedhosseini \cite[Theorem 5.2]{MR4300165}; the following
is just a slight repackaging of their results in our current language.
\begin{theorem}[Obstruction to PSC with Product Structure on the Boundary]
\label{thm:obstrprod}  
Suppose $(X, \p X)$ is a connected compact spin manifold
of dimension $n$ with non-empty boundary $\p X$. Let
$\Gamma=\pi(X)$ and $\Lambda=\pi(\p X)$.  Suppose
that $\Gamma$ and $\Lambda$ both satisfy injectivity of the
$KO$-assembly map $A_\Gamma\co KO_*(B\Gamma)\to KO_*(C^*(\Gamma))$ and
injectivity of the periodization map
$\per_\Gamma\co ko_*(B\Gamma)\to KO_*(B\Gamma)$
{\lp}and similarly for $\Lambda${\rp}.
Also assume injectivity of the periodization map
$\per_{\Gamma,\Lambda}\co ko_n(B\Gamma,B\Lambda)\to KO_n(B\Gamma,B\Lambda)$.
Then vanishing of the
class $c^{\Gamma,\Lambda}_*([X, \p X])\in ko_n(B\Gamma,B\Lambda)$
is necessary for $X$
to admit a Riemannian metric of {\psc} which is a product metric in
a neighborhood of the boundary.
\end{theorem}
\begin{proof}
  Note that by diagram \eqref{eq:koexactseq}, $c^{\Gamma,\Lambda}_*([X,\p X])$
  maps to $c^\Lambda_*([\p X])\in ko_{n-1}(B\Lambda)$, which since
  $A_\Lambda$ and $\per_\Lambda$ are injective, must vanish by Theorem
  \ref{thm:GLRmain} for $\p X$ to admit a {\psc} metric.  Thus, at a minimum,
  we know that $c^{\Gamma,\Lambda}_*([X,\p X])$ must be in the kernel of the
  boundary map, and thus must be the image of a class in
  $ko_n(B\Gamma)$.
  The rest of the proof can really be found in \cite{MR4045309} and
  \cite{MR4300165}, but we'll just restate the argument in slightly different
  form. Suppose that $X$ has been given a Riemannian metric $g_X$ which 
  is a {\psc} product metric $dt^2+g_{\p X}$ on a neighborhood of the boundary.
  We can attach a metric cylinder $\p X\times [0, \infty)$ to $X$
  along $\p X$, using the metric $dt^2+g_{\p X}$ on the cylindrical end.
  The noncompact manifold $\widehat X=X\cup_{\p X}\p X\times [0, \infty)$,
  equipped with the complete metric $g_{\widehat X}$ obtained by patching
  together $g_X$ and the product metric $dt^2+g_{\p X}$ on the cylindrical
  end, has uniformly positive scalar curvature off a compact set.
  Therefore $\Dirac_{\widehat X}$ is invertible off a compact set
  and thus Fredholm. Since $\widehat X$ is diffeomorphic to
  $\mathring X = X\smallsetminus \p X$  (it is just $\mathring X$
  with a collar attached to the boundary), $\Dirac_{\widehat X}$ together with
  the multiplication action of $C^{\bR}_0(\widehat X)$
  define a class in $KO_n(X,\p X)\cong KKO(C^{\bR}_0(\widehat X), \Cl_n)$,
  where $\Cl_n$ is the real Clifford algebra acting on the
  spinor bundle, and this class is an analytic representative for the class
  $\per_{X,\p X}[X,\p X]\in KO_n(X,\p X)$. For future use, note that
  the inclusion
  $\widehat X\hookrightarrow X\cup_{\p X}\p X\times [0, \infty]\simeq X$
  (where the final $\simeq$ denotes ``is homotopy equivalent to'')
  gives us a class $\widehat c_X\in KKO(C^{\bR}_0(\widehat X), C(B\Gamma))$.
  (If there is no finite model for $B\Gamma$, this is interpreted as
  $\varinjlim  KKO(C^{\bR}_0(\widehat X), C(Z))$, as $Z$ runs over
  finite subcomplexes of $B\Gamma$.)

  Now assume in addition that $g_X$ has {\psc} everywhere, and
  consider the Mishchenko-Fomenko index problem for $\Dirac_{\widehat X}$
  with coefficients in the Mishchenko-Fomenko flat $C^*(\Gamma)$-bundle
  over $\widehat X$.  Positivity of the scalar curvature and flatness
  of the bundle guarantee that the operator has vanishing
  index in the sense of $C^*(\Gamma)$-linear elliptic operators.
  By the Mishchenko-Fomenko Index Theorem, this $C^*$-index
  is the image of $[\Dirac_{\widehat X}]\in KO_n(X,\p X)$ under
  $\widehat c_X\co KO_*(X,\p X)\to KO_*(B\Gamma)$, followed by the assembly map
  $A_\Gamma$ (see for example \cite{MR842428}).  Since the
  assembly map is injective, we have $\widehat c_X([\Dirac_{\widehat X}])=0$
  in $KO_n(B\Gamma)$. Recall that we are trying to show that
  $c^{\Gamma,\Lambda}_*([X,\p X])=0$ in $ko_n(B\Gamma,B\Lambda)$.
  Since we are assuming $\per_{\Gamma,\Lambda}$ is injective, it suffices to show
  that $\per_{\Gamma,\Lambda}\circ c^{\Gamma,\Lambda}_*([X,\p X])=0$ in
  $KO_n(B\Gamma,B\Lambda)$. Now chase the commutative diagram:
  \begin{equation}
\label{eq:KOexactseq}
\xymatrix{
  ko_n(X) \ar[rr] \ar^{c^\Gamma_*}[dr] \ar^{\per_X}[dd] &
  & ko_n(X, \p X)  \ar^{c^{\Gamma,\Lambda}_*}[dr]
  \ar^(.7){\per_{X,\p X}}[dd] & \\
  &   ko_n(B\Gamma) \ar[rr]\ar^(.7){\per_\Gamma}[dd] & &
  ko_n(B\Gamma, B\Lambda)\ar^{\per_{\Gamma,\Lambda}}[dd] \\
  KO_n(X) \ar[rr] \ar^{c^\Gamma_*}[dr]  &
  & KO_n(X, \p X)  \ar^{c^{\Gamma,\Lambda}_*}[dr]   \ar^{\widehat c_X}[dl]& \\
  &   KO_n(B\Gamma) \ar[rr] & & KO_n(B\Gamma, B\Lambda).}
  \end{equation}
   
  The class whose vanishing we are trying to show lies in the group
  in the lower right, and is the same as $c^{\Gamma,\Lambda}_*([\Dirac_X])$,
  where $[\Dirac_X]\in KO_n(X,\p X)$.  Since this is the
  image of $\widehat c_X([\Dirac_X])=0$ in $KO_n(B\Gamma)$,
  this class vanishes, as desired.
\end{proof}

\begin{remark}
\label{rem:maxC*}
It might be useful to remark that the injectivity assumptions in Theorem
\ref{thm:obstrprod} on the periodization and assembly maps don't necessarily
have to hold in all degrees, just in the degrees where the relevant
indices appear ($n$ for $X$ and $n-1$ for $\p X$).

One can also repackage the argument in the proof using not just group
or groupoid {\Ca}s but also certain relative {\Ca}s.
A complication in doing that is that the
reduced group {\Ca} $C^*_r$ is functorial for injective group homomorphisms,
but not for surjective group homomorphisms.  Indeed, for some
discrete groups $G$, $C^*_r(G)$ is known to be simple (see
\cite{MR4174855} for a characterization of when this happens), hence there
cannot be a morphism $C^*_r(G)\to \bR$ corresponding to the
map of groups $G\twoheadrightarrow \{1\}$.  We can get around
this problem by defining $C^*(\Gamma, \Lambda)$ to be the mapping
cone (see \cite[\S1]{MR658514} of the map on full $C^*$-algebras
$C^*(\Lambda)=\bigoplus_j C^*(\Lambda_j)\to C^*(\Gamma)$
induced by the morphism of groupoids $\Lambda\to\Gamma$ coming from
the inclusion of $\p X$ into $X$.  Note that we are using the
\emph{maximal} groupoid $C^*$-algebra here, and are working with
\emph{real} {\Ca}s throughout.
Thus there is a canonical short exact sequence
\[
0\to C_0(\bR)\otimes C^*(\Gamma)\to C^*(\Gamma, \Lambda)
\to C^*(\Lambda)\to 0,
\]
making the $KO$-spectrum of $C^*(\Gamma, \Lambda)$ into the
homotopy fiber of the map of $KO$-spectra associated to
$C^*(\Lambda)\to C^*(\Gamma)$.  Then we can think of the idea of the
proof as dealing with vanishing of a ``relative index''
in $KO_n(C^*(\Gamma, \Lambda))$.  Note that injectivity of the
$KO$-assembly map for the maximal {\Ca} follows from injectivity of the
$KO$-assembly map for the reduced {\Ca}, and thus is automatic
for torsion-free groups satisfying the Baum-Connes Conjecture.
Injectivity of the $KO$-assembly map for the relative {\Ca} is
a more obscure condition, but it holds by diagram chasing if
$\Gamma$ and the $\Lambda_j$ are $K$-amenable and the assembly maps
for each of them are isomorphisms (for example if they are free, free abelian,
or surface groups).
\end{remark}

Putting Theorems \ref{thm:relko} and \ref{thm:obstrprod} together,
we obtain:
\begin{corollary}
\label{cor:NandSpscprod}
Suppose $(X, \p X)$ is a connected compact spin manifold
of dimension $n\ge 6$ with non-empty boundary $\p X$. Let
$\Gamma=\pi(X)$ and $\Lambda=\pi(\p X)$.  Suppose
that $\Gamma$, the $\Lambda_j$, and the pair $(\Gamma,\Lambda)$
satisfy both injectivity of the $KO$-assembly map
$A_\Gamma\co KO_*(B\Gamma)\to KO_*(C^*(\Gamma))$ and
injectivity of the periodization map
$\per_\Gamma\co ko_*(B\Gamma)\to KO_*(B\Gamma)$
{\lp}and similarly for the $\Lambda_j$ and the periodization map
for the pair{\rp}.  Then vanishing of the class
$c_*([X, \p X])\in ko_n(B\Gamma,B\Lambda)$ is necessary and sufficient for
$X$ to admit a Riemannian metric of {\psc} which is a product metric in
a neighborhood of the boundary.
\end{corollary}
\begin{proof}
  This is simply the amalgamation of Theorems
  \ref{thm:relko} and \ref{thm:obstrprod}.
\end{proof}
\begin{remark}
\label{rem:inj}
One might wonder how generally the hypotheses of Theorem
\ref{thm:obstrprod}  are valid.  The injectivity of the periodization map
usually has to be checked by an \emph{ad hoc} comparison of the
Atiyah-Hirzebruch spectral sequences for $ko_*$ and $KO_*$.
It usually fails in general for groups with torsion, but sometimes
it may hold in certain special dimensions.

For the hypothesis about injectivity of the assembly map, more general
techniques often apply for the groupoids $\Gamma$ and $\Lambda_j$, and then
for the relative groups one can apply the following simple lemma.
\end{remark}
\begin{lemma}
\label{lem:inj}
Suppose one has compatible splitting maps for $A_\Gamma$ and $A_\Lambda$,
i.e., there are splitting maps
$s_\Lambda\co KO_*(C^*(\Lambda))\to KO_*(B\Lambda)$ for $A_\Lambda$ and
$s_\Gamma\co KO_*(C^*(\Gamma))\to KO_*(B\Gamma)$ for $A_\Gamma$ such that
\[
\xymatrix{KO_*(B\Lambda) \ar[r] & KO_*(B\Gamma)\\
  KO_*(C^*(\Lambda)) \ar[r] \ar[u]^{s_\Lambda}
  & KO_*(C^*(\Gamma))\ar[u]^{s_\Gamma}}
\]
commutes.  Then the relative assembly map $A_{(\Gamma,\Lambda)}$
is injective. 
\end{lemma}
\begin{proof}
  Chase the diagram
\[
\xymatrix{
  KO_n(B\Lambda) \ar[r] \ar[d]^{A_\Lambda}
  & KO_n(B\Gamma) \ar[r] \ar[d]^{A_\Gamma}
  & KO_n(B\Gamma,B\Lambda) \ar[r]^\p \ar[d]^{A_{\Gamma,\Lambda}}
  & KO_{n-1}(B\Lambda) \ar[r] \ar[d]^{A_\Lambda} & \cdots \\
  KO_n(C^*(\Lambda)) \ar[r] \ar@/^/[u]^{s_\Lambda}
  & KO_n(C^*(\Gamma))\ar[r] \ar@/^/[u]^{s_\Gamma}
  & KO_n(C^*(\Gamma,\Lambda)) \ar[r]^\p &
  KO_{n-1}(C^*(\Lambda))\ar[r]  \ar@/^/[u]^{s_\Lambda} & \,\,\cdots\,\,.
}
\]
If $x\in KO_n(B\Gamma,B\Lambda)$ goes to $0$ under ${A_{\Gamma,\Lambda}}$,
then by commutativity of the right square and injectivity of
$A_\Lambda$, it maps under $\p$ to $0$ in $KO_{n-1}(B\Lambda)$, and hence
comes from a class $y\in KO_n(B\Gamma)$. But $y=s_\Gamma\circ A_\Gamma(y)$,
and $A_\Gamma(y)$ maps to $0$, so $A_\Gamma(y)$ is the image of some
$z\in KO_n(C^*(\Lambda))$.  By commutativity of the diagram in the
statement of the Lemma, $y$ is the image of $s_\Lambda(z)$,
and thus $x$, the image of $y$, must vanish.
\end{proof}
One can apply Lemma \ref{lem:inj} in the following context.  Suppose
all the fundamental groups belong to a class of groups for
which one can prove split injectivity of the $KO$-theory assembly
map in a functorial way.  There are many such classes, using ``dual
Dirac'' methods or embeddings into Hilbert spaces.  One may also
have to assume injectivity of the map $\Lambda\to \Gamma$
(at least if one just assumes that $\Gamma$ is coarsely embeddable into
a Hilbert space).  Then the hypothesis
of Lemma \ref{lem:inj} holds and one gets injectivity of the relative
assembly map as well.

\section{More Complicated Cases with Non-trivial Fundamental Groups}
\label{sec:gencase}
Now we move on to the more complicated situation, where $\p X$ is not
necessarily connected, or $\p X$ is connected but
the inclusion $\p X\hookrightarrow X$ does not induce
an isomorphism on fundamental groups.  In the latter
case, the conclusions
of Theorems \ref{thm:reloneconn}, \ref{thm:sc}, and \ref{thm:relko}
have to fail in some situations,
as one can see from the following simple example (mentioned also in
\cite{BH}).  Let $X = T^{n-2}\times D^2$, which has boundary the
$(n-1)$-torus $\p X= T^{n-1}$. The double of $X$ is
$M=\Dbl(X, \p X)=T^{n-2}\times S^2$, which obviously has a metric of
{\psc} (because of the $S^2$ factor). On the other hand, $\p X$ cannot
have a {\psc} metric.  It is also clear (since one can give $T^{n-2}$
a flat metric and give $D^2$ the metric corresponding to a spherical
cap around the north pole of a standard $2$-sphere, either close to
the pole or extending beyond the equator) that one can arrange for
$X$ to have a {\psc} metric so that the mean curvature $H$ of
$\p X$ is either strictly positive or strictly negative.  So for this
example, of the six conditions in Theorem \ref{thm:sc}, (1), (5), and (6)
hold, and (2) and (4) fail.  ((3) holds but does not say much,
since the $\alpha$-invariant gives insufficient information in the
non-simply connected case.)

\begin{definition}
\label{def:incompr}
Extending standard terminology from the case of $3$-manifolds,
we say that $X$ has \emph{incompressible boundary} if for each
component $\p_j X$ of $\p X$, the map $\pi_1(\p_j X)\to \pi_1(X)$
induced by the inclusion is injective.
\end{definition}

The incompressible case is the easiest case and also the one in which
we get the strongest conclusion.

\begin{theorem}
\label{thm:inccase}
Let $X$ be a connected compact spin manifold of dimension $n\ge 6$ with
non-empty incompressible boundary $\p X$. $X$ can have any number
$k$ of boundary components, $\p_1 X,\cdots,\p_k X$. Let $M=\Dbl(X,\p X)$.
Pick basepoints in each boundary component of $X$
and let $\Gamma=\pi_1(X)$ {\lp}this can be
with respect to any choice of basepoint{\rp}, $\Lambda_j=\pi_1(\p_j X)$,
and $\Lambda=\coprod_j \Lambda_j$.  Suppose that injectivity of the
periodization map $\per$ and of the $KO$-assembly map $A$ both
hold for the fundamental group $\Gamma*_{\Lambda}\Gamma$ of $M$.
Then the following conditions are
equivalent:
\begin{enumerate}
\item $M$ admits a metric of \psc.
\item $X$ admits a {\psc} metric which is a product metric in
  a neighborhood of the boundary $\p X$.
\item $X$ admits a {\psc} metric for which $\p X$ is minimal
  {\lp}i.e., has vanishing mean curvature{\rp}.
\item $X$ admits a {\psc} metric which gives $\p X$ positive mean curvature
  with respect to the outward normal.
\end{enumerate}
\end{theorem}
\begin{proof}
  As we have already mentioned, (2) $\Rightarrow$ (3) is trivial,
  (4) $\Rightarrow$ (1) is Theorem
  \ref{thm:GLdoubling} and (4) $\Longleftrightarrow$ (3) $\Rightarrow$ (1)
  is part of \cite[Corollary 34]{BH}. So it suffices to show
  that (1) $\Rightarrow$ (2), and we can try to apply Theorem
  \ref{thm:relko}. 
  Let $[M]$ be the $ko$-fundamental class of $M$ in $ko_n(M)$.
  This restricts (since $M=X\cup_{\p X}(-X)$ and thus
  $\text{int}\,X$ is an open submanifold of $M$)
  to the $ko$-fundamental class $[X,\p X]$ of $X$ in
  $ko_n(X, \p X) = ko_n(M, -X)$ (by excision).  This, in turn,
  maps under the $ko$-homology boundary map to the $ko$-fundamental
  class $[\p X]$ of $\p X$ in $ko_{n-1}(\p X)$. We get a commutative diagram
  \begin{equation}
    \label{eq:koclasses}
    \xymatrix{[M] \ar@{|->}[r] \ar@{}[d] |{\cap\!\!\text{\rule{.3pt}{5pt}}}
      & [X,\p X]\ar@{|->}[r] \ar@{}[d] |{\cap\!\!\text{\rule{.3pt}{5pt}}}
      & [\p X] \ar@{}[d] |{\cap\!\!\text{\rule{.3pt}{5pt}}}\\
      ko_n(M) \ar[r] \ar[d]^{c_M} & ko_n(X, \p X) \ar[r]^\p \ar[d]^{c_X}
      & ko_{n-1}(\p X)\ar[d]^{c_{\p X}}\\
      ko_n(B(\Gamma*_{\Lambda} \Gamma)) \ar[r] &
      ko_n(B\Gamma, B\Lambda)
      \ar[r]^\p       & \, ko_{n-1}(B\Lambda) .}
  \end{equation}
  The commutativity of the square on the right comes from naturality
  of the boundary maps, but the
  incompressibility assumption is needed to get commutativity of
  the square on the left.  Indeed, without this assumption, the commutativity
  must fail; just think of the simple example where $\Gamma$ is
  the trivial group and $\Lambda=\bZ$, and $X=D^2$ with boundary
  $\p X=S^1=B\Lambda$, $M=S^2$.  Then $B(\Gamma*_{\Lambda} \Gamma)=*$, and
  $[M]\mapsto 0$ in $ko_2(*)$ (since $M$ admits {\psc}), while
  $[X,\p X]$ is a generator of
  $ko_2(D^2,S^1)\cong ko_2(\bR^2)\cong ko_0(*)=\bZ$
  and $c_X\co ko_2(D^2,S^1)\to ko_2(*,S^1)$ is an isomorphism.

  So we need to explain why the left-hand square in \eqref{eq:koclasses}
  commutes when we assume incompressibility.  The explanation is
  that $\pi_1(M)$ is the pushout or colimit of the diagram
  $\xymatrix{\Gamma&\Lambda\ar[l]\ar[r]&\Gamma}$ (in the category of
  groupoids), so $B\pi_1(M)$ is the \emph{homotopy pushout} (hocolim)
  of $\xymatrix{B\Gamma&B\Lambda\ar[l]\ar[r]&B\Gamma}$ in spaces.
  This is usually \emph{not} the same as the ordinary pushout since
  the kernel of $\Lambda_j\to\Gamma$ is invisible in the homotopy
  colimit, but it \emph{is} the same when each $\Lambda_j$ injects into
  $\Gamma$, i.e., $\p X$ is incompressible. (See
  \cite[\S10]{MR1361887} for an explanation.)   So in that case we
  get the commutativity via excision for $ko$, since
  $B(\Gamma*_{\Lambda} \Gamma)$ with $B(\Lambda*_{\Lambda} \Gamma)$
  collapsed is the the same as $B\Gamma$ with $B\Lambda$
  collapsed, and hence the relative groups
  $ko_n(B(\Gamma*_{\Lambda} \Gamma),B(\Lambda*_{\Lambda} \Gamma))$
  and $ko_n(B\Gamma,B\Lambda)$ are the same in this case.

  Alternatively, following the point of view in
  \cite[Appendix to Chapter II]{MR1324339}, note that $\pi_1(M)$ is
  associated to the graph of groups shown in Figure \ref{fig:graphgroups}.
  More precisely, the universal cover of this graph is a tree on
  which $\pi_1(M)$ acts, with $\Gamma$ stabilizing the vertices
  and with the edges stabilized by the $\Lambda_j$.  By Bass-Serre
  theory, this action on a tree corresponds to the decomposition
  of $\pi_1(M)$ as $\Gamma*_\Lambda\Gamma$, and displays $\pi_1(M)$ as
  an extension of $\Gamma$ by a free group.  From this picture
  we can also read off $B(\Gamma*_{\Lambda} \Gamma)$ as the
  ordinary pushout of $\xymatrix{B\Gamma&B\Lambda\ar[l]\ar[r]&B\Gamma}$.
  
  \begin{figure}[hbt]
  \[
  \xymatrix{\Gamma\,\bullet \ar@{-} @/^1pc/[r]^{\Lambda_1}
    \ar@{-}[r]|{\text{  }\raisebox{5pt}{$\vdots$}\text{  }}
    \ar@{-} @/_1pc/[r]_{\Lambda_k}
   &  \bullet\,\,\Gamma} 
  \]
  \caption{\label{fig:graphgroups} The graph of groups for computing
    $\pi_1(M)$} 
  \end{figure}
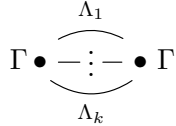 
  Now because of the hypothesis on $\pi_1(M)$ and the assumption that $M$
  admits a metric of {\psc}, we can apply \cite[Theorem 4.13]{MR1818778},
  and deduce that the image of $[M]$ under $c_M$ (the leftmost
  downward arrow in \eqref{eq:koclasses}) vanishes in
  $ko_n(B(\Gamma*_{\Lambda} \Gamma))$.  Chasing the diagram, we see
  that $(c_X)_*([X,\p X])$ and $(c_{\p X})_*([\p X])$ also vanish, so by
  Theorem \ref{thm:relko}, $X$ admits a {\psc} metric which is a product
  metric in a neighborhood of $\p X$.
\end{proof}

We don't quite get the same sort of theorem in the totally non-spin case
since we don't know what all the obstructions are (if any) in this situation
where there is no Dirac operator.  However, we can at least
prove the following.
\begin{theorem}
\label{thm:inccasenonspin}
Let $X$ be a connected compact totally non-spin manifold of dimension
$n\ge 6$ with non-empty incompressible boundary $\p X$, and let
$M=\Dbl(X,\p X)$. There can be any number $k$ of boundary
components $\p_1 X,\cdots,\p_k X$, but assume that each of these
is also totally non-spin.  Pick basepoints in each boundary
component and let $\Gamma=\pi_1(X)$ {\lp}this can be
with respect to any choice of basepoint{\rp}, $\Lambda_j=\pi_1(\p_j X)$.
Suppose that $\Gamma$ has finite homological dimension and that
$n\ge\homdim\Gamma+2$.
Then the following conditions all hold:
\begin{enumerate}
\item $M$ admits a metric of \psc.
\item $X$ admits a {\psc} metric which is a product metric in
  a neighborhood of the boundary $\p X$.
\item $X$ admits a {\psc} metric for which $\p X$ is minimal
  {\lp}i.e., has vanishing mean curvature{\rp}.
\item $X$ admits a {\psc} metric which gives $\p X$ positive mean curvature
  with respect to the outward normal.
\end{enumerate}
\end{theorem}
\begin{proof}
  The description of $\pi_1(M)$ is just as in the proof of
  Theorem \ref{thm:inccase}, so $\Gamma*_\Lambda \Gamma$
  splits as $F\rtimes \Gamma$, where $F$ is a free group.
  That means $\homdim \pi_1(M) \le \homdim\Gamma + 1$.
  Since the $\Lambda_j$'s are subgroups of $\Gamma$, they also have
  finite homological dimension bounded by $\homdim\Gamma$.
  Also, by the long exact sequence of the pair,
  $\homdim(B\Gamma,B\Lambda)\le \max(\homdim\Gamma,\homdim\Lambda+1)< n$.
  So by our assumption on $n$, $c_*[M]=0$ in
  $H_n(B\pi_1(M);\bZ)$, $c_*([\p_j X])=0$ in $H_{n-1}(B\Lambda;\bZ)$,
  and $c_*([X,\p X])=0$ in $H_n(B\Gamma,B\Lambda;\bZ)$. 
  Since $M$ and the $\p_j X$ are all totally-non-spin, we conclude
  by \cite[Theorem 4.11]{MR1818778} or by  \cite[Theorem 1.2]{MR3078256}
  that all of them admit metrics of {\psc}.  Then we obtain
  (2), and thus also (3) and (4), by Theorem \ref{thm:relhom}.
\end{proof}

Now we go on to the non-incompressible case, as well as
to the case where $X$ is not spin but has spin boundary.
In the latter case, even if $X$ and $\p X$ are both simply
connected, we cannot expect $X$ to have a {\psc} metric which
is a product metric near the boundary.  For example, $\p X$
could be an exotic sphere in dimension $\equiv1$ or $2$ mod $8$
with non-zero $\alpha$-invariant.  Then $\p X$ is an oriented
boundary but not a spin boundary, and thus there is an oriented
manifold (which can be chosen simply connected) having $\p X$ as
its boundary. The double of this manifold is simply connected
and non-spin, so it admits {\psc} by \cite[Corollary C]{MR577131}.
In this case, by Theorem \ref{thm:MichelsohnLawson},
$X$ has a {\psc} metric which has positive
mean curvature on the boundary, even though it can't have a
{\psc} metric which is a product metric near the boundary, since $\p X$
does not admit {\psc}. We will see other generalizations of this later.
First we consider the non-incompressible spin case of the converse to
Theorem \ref{thm:GLdoubling}.
\begin{theorem}
\label{thm:gencase}
Let $X$ be a connected compact manifold of dimension $n\ge 6$ with
non-empty boundary $\p X$. There can be any number $k$ of boundary
components $\p_1 X,\cdots,\p_k X$.  Assume that $X$ is spin and
that $\Gamma=\pi_1(X)$ has finite homological dimension less than $n$
and satisfies the Baum-Connes Conjecture with coefficients and the
Gromov-Lawson-Rosenberg conjecture.
{\lp}As an example, $\Gamma$ could be free abelian, free, or a
surface group.{\rp} Let $M=\Dbl(X, \p X)$ be the double
of $X$ along $\p X$, a closed $n$-manifold. Then the following conditions
are equivalent:
\begin{enumerate}
\item $M$ admits a metric of \psc.
\item $X$ admits a {\psc} metric for which $\p X$ is minimal
  {\lp}i.e., has vanishing mean curvature{\rp}.
\item $X$ admits a {\psc} metric which gives $\p X$ positive mean curvature
  with respect to the outward normal.
\end{enumerate}
\end{theorem}
\begin{proof}
  We have $\pi_1(M)=\Gamma*_\Lambda \Gamma$ by Van Kampen's Theorem.
  (This means the colimit of the diagram
  $\xymatrix{\Gamma&\Lambda\ar[l]\ar[r]&\Gamma}$.)
  We have the ``folding map'' $f\co M\twoheadrightarrow X$,
  sending each point $x\in X$ and the corresponding point $\bar x\in -X$
  both back to $x$, which is split by the inclusion of $X$ into $M$,
  and thus an induced split surjection of groups
  $f_*\co \Gamma*_\Lambda \Gamma\to \Gamma$.  
  By lifting to the universal cover of $X$ and arguing as in the
  proof of Theorem \ref{thm:sc}, we see that $\Gamma*_\Lambda \Gamma$
  splits as $F\rtimes \Gamma$, where $F$ is a free group.
  Thus if we know Baum-Connes with coefficients for $\Gamma$,
  the same follows for $\pi_1(M)=\Gamma*_\Lambda \Gamma$.
  Furthermore, the homological dimension of $\pi_1(M)$ is at most
  one more than the homological dimension of $\Gamma$, and is thus
  $\le n$.  This implies injectivity of the periodization maps
  $ko_n(B\Gamma)\to KO_n(B\Gamma)$ and $ko_n(B\pi_1(M))\to KO_n(B\pi_1(M))$,
  because any differentials in the Atiyah-Hirzebruch spectral
  sequences $H_p(B\Gamma, ko_q)\Rightarrow ko_{p+q}(B\Gamma)$
  and $H_p(B\Gamma, KO_q)\Rightarrow KO_{p+q}(B\Gamma)$ affecting the
  line $p+q=n$ have to be the same (and similarly for $B\pi_1(M)$),
  because $H_p$ vanishes for $p>n$.  (See Figure \ref{fig:AHSS}.)
  Since $\Gamma$ and $\pi_1(M)$
  both have finite homological dimension, they are torsion-free and
  Baum-Connes implies that the $KO$ assembly maps for both are
  injective.

  \begin{figure}[hbt]
    \begin{tikzpicture}
      \fill[blue!10!white] (0,0) -- (3,0) -- (3,3) -- (0,3) -- (0,0);
      \draw[thick,black,->] (0,0) -- (3.5,0);
      \draw[thick,red,->] (1,2) -- (1.8,1.5);
      \draw[thick,red,->] (1.2,1.5) -- (2,1);
      \node at (3.7, 0) {$p$};
      \node at (3, -.3) {$n$};
      \node at (-.3, 3) {$n$};
      \fill (3,0) circle (2pt);
      \fill (0,3) circle (2pt);
      \draw[thick,black,->] (0,0) -- (0,3.5);
      \draw[thick,black,dashed] (0,3) -- (3,0);
      \node at (0, 3.7) {$q$}; 
    \end{tikzpicture}  
  \caption{\label{fig:AHSS} The Atiyah-Hirzebruch spectral sequence
    for computing $ko_n(B\pi_1(M))$.  The corresponding spectral sequence
    for $KO_n(B\pi_1(M))$ is similar except that it extends into the
    fourth quadrant, and the map from this spectral sequence to that
    one preserves differentials affecting the line $p+q=n$.  Typical
    differentials are shown in color.}
  \end{figure}
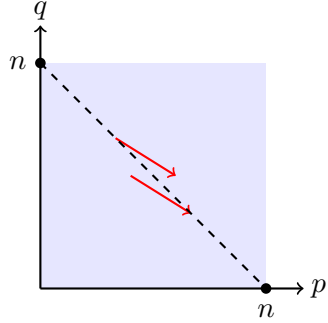

  As before, it suffices to prove that (1) $\Rightarrow$ (3).
  So assume $M$ has a metric of {\psc}.
  It will be convenient to denote by $\Lambda'$ the image of
  $\Lambda$ in $\Gamma$, so that also $\pi_1(M)=\Gamma*_{\Lambda'} \Gamma$.
  Start by choosing a collar neighborhood (diffeomorphic to
  $\p X \times [0, 1]$) of $\p X$ inside $X$ and let $\p' X=\coprod \p'_j X$
  be a ``parallel copy'' of $\p X$ inside this collar.
  First assume that for each $j$,
  $K_j=\ker\bigl(\Lambda_j\twoheadrightarrow \Lambda'_j)$ is finitely
  generated.  Then we modify each $\p'_j X$ to form a new hypersurface $Z_j$,
  by doing a finite number of surgeries on embedded circles to kill off the
  kernel $K_j$, inserting the necessary $2$-handles inward in $X$,
  as in Figure \ref{fig:innersurg}.  Note that the fact that $K_j$
  dies in $\pi_1(X)$ is what guarantees that we can build the necessary
  $2$-handles inside $X$.

  \begin{figure}[hbt]
    \begin{tikzpicture}
      \draw[very thick,black] (0cm, 1cm) arc(90:270:2cm and 1cm);
      \draw[very thick,black] (0cm, 1cm) -- (.7cm, 1cm);
      \draw[very thick,black] (0cm, -1cm) -- (.7cm, -1cm);
      \draw[very thick,red] (0cm, -1cm) arc(-90:90:.6cm and 1cm);
      \draw[very thick,red] (.7cm, -1cm) arc(-90:90:.6cm and 1cm);
      \node at (-2.3cm, 0cm) {$X$};
      \node at (1cm, 0cm) {$\p'_j X$};
      \node at (1.7cm, 0cm) {$\p_j X$};
      \node at (2cm, 1cm) {collar};
      \draw[green,snake=snake,->] (1.8cm,.7cm) -- (1cm,0.4cm);
      \draw[very thick,green] (.59cm, .1cm) arc(90:270:.2cm and .1cm);
      \draw[green,snake=snake,->] (-0cm,.2cm) -- (0.35cm,0.13cm);
      \draw[very thick,green] (.5cm, -.5cm) arc(60:280:.2cm and .1cm);
      \draw[green,snake=snake,->] (-0cm,-.02cm) -- (0.3cm,-0.37cm);
      \node at (-.8cm, .2cm) {2-handle};
    \end{tikzpicture}  
    \caption{\label{fig:innersurg} First step in the proof
      of Theorem \ref{thm:gencase}.  $Z_j$ is the result of attaching
      $2$-handles to $\p'_j X$ to reduce the fundamental group
      to $\Lambda_j'$.}
  \end{figure}
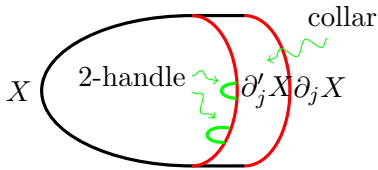

  After this initial surgery step, we have a new manifold $Y$ with
  boundary, of codimension $0$ in $X$, with $\p Y = Z = \coprod Z_j$
  incompressible in $X$.  Thus {\psc} for $M$ implies via the proof
  of Theorem \ref{thm:inccase} that $Y$ has a metric of {\psc} which
  is a product metric in neighborhood of $\p Y$.
  We double the region between $\p Y$ (which we recall is $\p' X$
  modified by surgeries) and $\p X$ across $\p X$, obtaining
  a manifold $W$ with boundary of codimension $0$ in $M$, whose boundary
  looks like two copies of $\p Y$, and is thus incompressible.
  (See Figure \ref{fig:outersurg}.)  Since the fundamental
  group of $\p Y$ is, by construction, a quotient of that of $\p X$,
  Van Kampen's Theorem gives that $\pi_1(\p Y)\cong \pi_1(W)$, with
  the isomorphism induced by the inclusion.  Thus $W$ is a
  $\pi_1$-preserving spin
  cobordism over $B\pi_1(\p Y)$ from $\p Y$ to its reflection
  across $\p X$. (See Figure \ref{fig:outersurg} again.) 
  By the surgery theorem for {\psc} metrics (\cite{MR577131} and
  \cite[Corollary 6.2]{MR4236952}), we can extend the metric on
  $Y$ to a {\psc} metric on $Q = Y\cup_{\p Y} W$ which is a product metric
  in a neighborhood of $\p Q$ (the reflection of the
  original $\p Y$ across $\p X$).  At this point, $\p Y$
  is now a two-sided totally geodesic hypersurface in $Q$.  By a slight
  deformation, using \cite[Proposition 28]{BH}, we can preserve
  {\psc} on $Q$ but arrange for $\p Y$ to have positive mean curvature.
  Now observe that $\p X\subset Q$ is obtained by adding handles
  back to $\p Y$, the duals of the
  surgeries used to construct $Y$ from $X$, and these are in codimension
  $2$.  So we can apply Theorem \ref{thm:handleattach} to
  deform $\p X$ so that it has positive mean curvature.  Since
  $X$ is inside $Q$, which has {\psc}, we have
  proven that (1) $\Rightarrow$ (3).

    \begin{figure}[hbt]
    \begin{tikzpicture}
      \fill[red!20!white] (-0.5cm, -.95cm) -- (-0.5cm, .95cm)
      arc(105:255:2cm and 1cm);
      \fill[red!35!white] (-0.5cm, .95cm) -- (-0.5cm, -.95cm)
      -- (0.5cm, -.95cm) -- (0.5cm, .95cm) -- (-0.5cm, .95cm);
      \fill[red!35!white] (.5cm, .5cm) arc(-90:90:.2cm and .1cm) --
      (.5cm, .5cm);
      \fill[red!35!white] (.5cm, -.7cm) arc(-90:90:.2cm and .1cm) --
      (.5cm, -.7cm);
      \fill[red!35!white] (.5cm, -.1cm) arc(-90:90:.2cm and .1cm) --
      (.5cm, -.1cm);
      \fill[red!35!white] (-.5cm, .7cm) arc(90:270:.2cm and .1cm)
      -- (-.5cm, .7cm);
      \fill[red!35!white] (-.5cm, -.5cm) arc(90:270:.2cm and .1cm)
      -- (-.5cm, -.5cm);
      \fill[red!35!white] (-.5cm, .1cm) arc(90:270:.2cm and .1cm)
      -- (-.5cm, .1cm);
      \node at (-1cm, 0cm) {$Y$};
      \node at (-.9cm, .6cm) {$\scriptstyle \p Y$};
      \node at (.2cm, 0cm) {$W$};
      \node at (.9cm, .4cm) {$\scriptstyle \p W$};
      \node at (-.2cm, .5cm) {$\scriptstyle \p X$};
      \draw[very thick,black] (2cm, 0cm) arc(0:360:2cm and 1cm);
      \draw[very thick,black] (0cm, -1cm) -- (0cm, 1cm);
      \draw[very thick,black] (0.5cm, -.95cm) -- (0.5cm, .95cm);
      \draw[very thick,black] (-0.5cm, -.95cm) -- (-0.5cm, .95cm);
      \draw[very thick,green] (.5cm, .5cm) arc(-90:90:.2cm and .1cm);
      \draw[very thick,green] (-.5cm, .7cm) arc(90:270:.2cm and .1cm);
      \draw[very thick,green] (.5cm, -.7cm) arc(-90:90:.2cm and .1cm);
      \draw[very thick,green] (-.5cm, -.5cm) arc(90:270:.2cm and .1cm);
      \draw[very thick,green] (.5cm, -.1cm) arc(-90:90:.2cm and .1cm);
      \draw[very thick,green] (-.5cm, .1cm) arc(90:270:.2cm and .1cm);
    \end{tikzpicture}  
    \caption{\label{fig:outersurg} Second step in the proof
      of Theorem \ref{thm:gencase}.  First we construct a {\psc} metric
      on $Y$ (light color), then push it across $W$ (dark color),
      and then get back to $\p X$ from $\p Y$ by doing outward-pointing
      surgeries.}
  \end{figure}
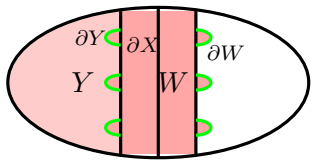

  There is just one more step if one or more of the kernels $K_j$
  is not finitely generated.  The problem now is that it appears we
  need to do infinitely many surgeries to get from $X$ to $Y$ and
  back again.  However, we can construct a sequence of modifications
  of $\p_j X$ (each obtained from the previous one by attaching more
  $2$-handles into the interior of $X$), say $Z_j^{(k)}$, $k=1,2\cdots$,
  so that $\pi_1(Z_j^{(k)})\to \Lambda'_j$ as $k\to \infty$.  We get
  a sequence $X^{(k)}$ of manifolds with boundary,
  $\p X^{(k)}=\coprod_j Z_j^{(k)}$, and $ko$-homology
  classes, the images of $[X, \p X]$ in
  $ko_n(X^{(k)}, \p X^{(k)})$ which tend to $0$ in the limit.  By the
  behavior of homology under inductive limits, there must be some finite
  stage at which $[X^{(k)}, \p X^{(k)}]$ vanishes, and thus
  $X^{(k)}$ has a {\psc} metric which is a product metric near the
  boundary.  The proof is now concluded as before.
\end{proof}

The analogous theorem in the totally non-spin case is this:
\begin{theorem}
\label{thm:gencasenonspin}
Let $X$ be a connected compact manifold of dimension $n\ge 6$ with
non-empty boundary $\p X$. There can be any number $k$ of boundary
components $\p_1 X,\cdots,\p_k X$.  Assume that $X$ is totally non-spin.
Pick basepoints in each boundary
component and let $\Gamma=\pi_1(X)$ {\lp}this can be
with respect to any choice of basepoint{\rp}, $\Lambda_j=\pi_1(\p_j X)$.
Assume that for each $j$, the image of $\Lambda_j$ in $\Gamma$ is
finitely presented. Suppose that $\Gamma$ has finite homological
dimension and that $n\ge\homdim\Gamma + 2$. Let $M=\Dbl(X, \p X)$ be
the double of $X$ along $\p X$, a closed $n$-manifold. Then the
following conditions hold:
\begin{enumerate}
\item $M$ admits a metric of \psc.
\item $X$ admits a {\psc} metric for which $\p X$ is minimal
  {\lp}i.e., has vanishing mean curvature{\rp}.
\item $X$ admits a {\psc} metric which gives $\p X$ positive mean curvature
  with respect to the outward normal.
\end{enumerate}
\end{theorem}
\begin{proof}
  Note that $M$ is totally non-spin and its dimension $n$ exceeds the
  homological dimension of $\pi_1(M)=\Gamma*_\Lambda\Gamma$, which
  is no more than $\homdim\Gamma + 1$, as in the proof of Theorem
  \ref{thm:inccasenonspin}.  So by 
  \cite[Theorem 4.11]{MR1818778} or \cite[Theorem 1.2]{MR3078256},
  $M$ admits a metric of {\psc},
  verifying (1).  To check (2) and (3), we apply the same argument
  as in Theorem \ref{thm:gencase}; i.e., we first do surgeries
  into the interior to kill off the kernel of $\Lambda_j\to\Gamma$
  for each $j$, ending up with a connected totally
  non-spin manifold $Y$ of codimension $0$ in $X$
  with incompressible boundary. Now recall
  that $X$ is totally non-spin, and so is $Y$, but some of the boundary
  components $\p_j Y$ might ``accidentally'' fail to be
  totally non-spin.  If this is
  the case for some $j$, do an additional surgery to take a connected
  sum of $\p_j Y$ with a tubular neighborhood of an embedded $2$-sphere
  in the interior of
  $Y$ whose normal bundle is non-trivial (such $2$-spheres exist
  since $Y$ is totally non-spin).  After doing this, all the $\p_j Y$ will
  also be totally non-spin, and hence will admit metrics of {\psc} for
  the same reason as $M$ (since the condition on the homological
  dimension of $\Gamma$ inherits to subgroups). Now we can apply Theorem
  \ref{thm:inccasenonspin} to $Y$,
  and then apply Theorem \ref{thm:handleattach} to get positive mean
  curvature on the boundary of $X$.
\end{proof}
\begin{remark}
\label{rem:nonspinntfg}  
By the same reasoning as in the proof of Theorem \ref{thm:gencase},
one can dispense with the finite presentation hypothesis on the images
of $\Lambda_j$ in $\Gamma$.  Even though, when one of these images
is not finitely presented, it won't be possible 
to modify $Y$ in finitely many steps
so as to have incompressible boundary, $c_*([Y,\p  Y])$
vanishes in $H_n(B\Gamma, B\operatorname{image}(\Lambda);\bZ)=0$,
and thus will vanish in $H_n(B\pi_1(Y), B\pi_1(\p Y);\bZ)$ for $Y$
``close enough'' to having incompressible boundary, and then we can
proceed as before.
\end{remark}

\section{The Gromov-Lawson Doubling Theorem}
\label{sec:GL}

Since the proof of Theorem \ref{thm:GLdoubling} in \cite{MR569070}
is somewhat sketchy, we include for the convenience of the reader
a more detailed proof.  We should mention that the paper \cite{BH} also
gives another approach to ``doubling.'' Indeed, \cite[Corollary 34]{BH} gives
a homotopy equivalence between $\cR^+(X)_{H>0}$ ({\psc} metrics on $X$
with positive mean curvature on $\p X$) and $\cR^+(X)_{\Dbl}$ ({\psc} metrics
on $X$ with vanishing second fundamental form on the boundary, that extend
to reflection-invariant {\psc} metrics on $\Dbl(X,\p X)=M$).

\begin{proof}[{Proof of Theorem \ref{thm:GLdoubling}}]
  Following the proof in \cite{MR569070}, let $I = [-1,1]$, give $X\times I$
  the product metric, and identify $X$ with $X\times \{0\}$.  Let
  \[
  N=\{(x,t)\in X\times I: \text{dist}\,((x,t),X_1)\le\varepsilon\}
  \subset X\times \bR,
  \]
  where $X_1$ is the complement in $X$ of a small open collar around $\p X$.   
  Then the double $M$ is obtained by smoothing the $C^1$ manifold
  $\p N$. We give $N$ the metric inherited
  from $X\times I$. The portions of $\p N$ given by $X_1\times \{\varepsilon\}$
  and $X_1\times \{-\varepsilon\}$ are smooth and isometric
  to $X_1$, but the second derivative of the metric is discontinuous
  at points in $\p X_1\times \{\pm \varepsilon\}$ like $z$ in
  Figure \ref{fig:fig1}. As in \cite{MR569070}, 
  choose $x\in \p X_1$ and let $\sigma$ be a small geodesic segment
  in $X$ passing through $x$ and orthogonal to the boundary of $X_1$.  Then
  $(\sigma \times I)\cap N$ is flat and totally geodesic in $N$, and a local
  picture of it near $x$ looks like Figure \ref{fig:fig1}.
  In this plane we have the unit vector field $\boldsymbol r$ (`r' for right)
  which along $\sigma\times \{t\}$ is the unit tangent vector pointing away
  from the interior of $X_1$.  We also have the unit vector field
  $\boldsymbol d$ (`d' for down) pointing downward along the
  lines $\{m\}\times \bR$, $m\in X$.  If we instead take
  a slice of $N$ parallel to $X_1$ through a point $y$ at angle $\theta$
  from $X_1$ (see Figure  \ref{fig:fig1} and Figure
  8 in \cite{MR569070}), we get a picture like Figure \ref{fig:fig2}.
  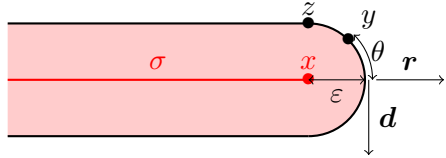
\begin{figure}[hbt]
    \begin{tikzpicture}
      \fill[red!20!white] (-1, .75) -- (-1, -.75) --
      (3, -0.75) arc (-90:90:0.75) -- (-1, .75);
      \draw[thick,red]  (-1,0) -- (3, 0)
      node[above,midway] {$\sigma$} node[above, at end] {$x$};
      \node[red] at (3,0) {$\bullet$};
      \draw[black, <->] (3, 0.0) -- (3.75, 0.0)
      node[below,midway] {$\varepsilon$};
      \draw[thick,black] (-1,0.75) -- (3, 0.75);
      \draw[thick,black] (-1,-0.75) -- (3, -0.75);
      \draw[thick,black] (3, -0.75) arc (-90:90:0.75);
      \draw[black,->] (3.8,0) -- (3.8,-1) node[right,midway]
           {$\boldsymbol d$};
      \draw[black,->] (3.9,0) -- (4.8, 0) node[above,midway]
           {${\boldsymbol r}$};
      \draw[black, <->] (3.85, 0) arc (0:45:0.85);
      \node at (3.924, .382) {$\theta$};
      \node at (3, 0.75) {$\bullet$};
      \node at (3.53, .53) {$\bullet$};
      \node at (3, .95) {$z$};
      \node at (3.8, .8) {$y$};
    \end{tikzpicture}  
  \caption{\label{fig:fig1} Slice of $N$ normal to $X_1$}
  \end{figure}
  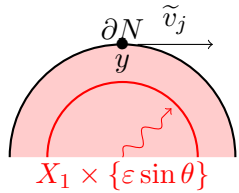
\begin{figure}[hbt]
    \begin{tikzpicture}
      \fill[red!20!white] (1.5,0) arc(0:180:1.5) -- (1.5, 0);
      \draw[thick,red]  (1,0) arc(0:180:1);
      \draw[thick,black]  (1.5,0) arc(0:180:1.5);
      \node[red] at (0,-.25) {$X_1\times \{\varepsilon\sin\theta\}$};
      \draw[red,snake=snake,->] (0,0) -- (0.65,0.65);
      \draw[black,->] (.2,1.5) -- (1.2,1.5)
      node[above,midway] {$\widetilde v_j$};
     \node at (0, 1.5) {$\bullet$};
      \node at (0, 1.25) {$y$};
      \node at (0, 1.7) {$\p N$};
    \end{tikzpicture}  
  \caption{\label{fig:fig2} Slice of $N$ through $y$ parallel to $X_1$}
  \end{figure}
  
  Choose an orthonormal
  frame $v_1,\cdots, v_{n-1}$ for $T_x(\p X_1)$ that diagonalizes
  the second fundamental form for $\p X_1$ in $X_1$ with respect to the
  outward-pointing normal vector vector field $\boldsymbol n$,
  which at $x$ coincides with $\boldsymbol r$.
  Thus we can assume that the shape operator has the form
  $S_{\p X_1}v_j=\mu_j v_j$, and the mean curvature $H_{\p X_1}$
  of $\p X_1$ is $\sum_j \mu_j$.  The assumption of the theorem implies
  that this is positive.  (Positive mean curvature of $\p X$ in $X$
  implies positivity of $H_{\p X_1}$ if $X_1$ is
  close enough to $X$.) At a point $y\in \p N$ at distance $\varepsilon$
  from $x\in \p X_1$ as in Figure \ref{fig:fig1}, the outward-pointing
  normal vector to $N$ is given by
  $\boldsymbol n=\cos\theta\,\boldsymbol r-\sin \theta\,\boldsymbol d$.
  There is an orthonormal frame
  $w_0,w_1,\cdots, w_{n-1}$ for $T_y(\p N)$ with
  $w_0 = \cos \theta \,\boldsymbol d + \sin\theta\, \boldsymbol r$,
  and with $w_j$ close to the parallel transport $\widetilde v_j$
  of $v_j$.  Since $\sigma\times \bR$ is totally geodesic in $X\times \bR$
  and $\p N\cap (\sigma\times \bR)$, shown in Figure \ref{fig:fig1},
  consists of $\sigma\times\{\pm\varepsilon\}$ joined together by
  a semicircular arc of radius $\varepsilon$, it follows that $w_0$, which
  lies in this $2$-plane, is
  an eigenvector for the shape operator of $\p N$, with eigenvalue
  $\frac{1}{\varepsilon}$.  On the other hand, for $j=1,\cdots, {n-1}$, 
  $\nabla_{w_j}\boldsymbol d=0$ and so
  \[
    S_{\p N}(w_j) = -\nabla_{w_j}\boldsymbol n\approx
    \cos\theta\, \mu_jw_j
  \]
  and so by the Gauss curvature formula, the scalar curvature
  $\kappa_{\p N}$ of $\p N$ at $y$ works out to
  \[
  \begin{aligned}
  \kappa_{\p N} & \approx \kappa_X +
  \sum_{j=1}^{n-1}\frac{2}{\varepsilon}\cos\theta\,\mu_j
  + \sum_{1\le j\ne  k\le n-1}\cos^2\theta\,\mu_j\mu_k\\
  &=\kappa_X +\frac{2\cos\theta}{\varepsilon}\,H_{\p X_1} +
  \cos^2\theta\,H_{\p X_1}^2 - \sum_{j=1}^{n-1}\cos^2\theta\,\mu_j^2.
  \end{aligned}
  \]
  Since $\kappa_X$ and $H_{\p X_1}$ are positive, this is positive
  and tends to $\kappa_X$ as $\theta\to \pm\frac{\pi}{2}$.
  So we can round the corners at these points keeping positivity of
  the scalar curvature.  (Note that the formula obtained here is slightly
  different from the one in \cite{MR569070}, but this doesn't affect the
  conclusion.)
\end{proof}

\section{Open Problems}
\label{sec:questions}

We have left several open problems in our discussion.  In this section,
we list a few of these and say something about where they stand.
\begin{enumerate}
\item The most obvious question is how generally a converse to
  Theorem \ref{thm:GLdoubling} is valid.  One can state this as
  \begin{conjecture}[``Doubling Conjecture'']
    \label{conj:doubling}
    If $X$ is a compact manifold with boundary and $M=\Dbl(X, \p X)$
    admits a metric of {\psc}, then $X$ admits a metric of {\psc} with
    positive mean curvature on $\p X$.
  \end{conjecture}
  At the moment we do not know of any counterexamples, nor do we know
  of any technology that could be used to disprove this in general.
  Conjecture \ref{conj:doubling} has now been proved in dimension $3$ by
  Carlotto and Li \cite{carlotto2021,carlotto2021a}.  Note by the
  way that \cite[\S2.3]{BH} gives an obstruction to a manifold $X$
  with boundary admitting a {\psc} metric with $H>0$ on the boundary,
  but it can't give any counterexamples to the Doubling Conjecture
  since in all cases where the obstruction applies, it applies
  to the double as well.
\item Another question is how generally Corollary \ref{cor:NandSpscprod}
  and Theorem \ref{thm:relbordismnonspin} can hold.  One cannot
  get rid of the dimension restriction, since
  Theorem \ref{thm:relhom} fails
  when $\dim X = 5$, as one can see from the following counterexample.
  Let $Y$ be a smooth non-spin simply connected projective algebraic surface
  (over $\bC$) of general type, for example a hypersurface of degree $d$
  in $\bC\bP^3$ of even degree $d\ge 4$.  Then
  $b_2^+(Y)=1+\frac{(d-1)(d-2)(d-3)}{3}>1$ and $Y$ has non-trivial
  Seiberg-Witten invariants, hence does not admit a metric of {\psc}.
  Let $X=Y\times [0,1]$, which is a compact $5$-manifold with boundary.
  The double $M$ of $X$ along its boundary is $Y\times S^1$, which is
  a closed totally non-spin $5$-manifold with fundamental group $\bZ$,
  and which thus admits {\psc} as a consequence of the theorem
  of Stolz and Jung (\cite[Theorem 4.11]{MR1818778}
  or \cite[Theorem 1.2]{MR3078256}).  However $X$ cannot admit a
  {\psc} metric which is a product metric near the boundary, since
  $\p X = Y \sqcup -Y$ does not admit a metric of {\psc}.
  Yet the relative homology group $H_5(B*, B(*\sqcup *);\bZ)$ vanishes.

  One can modify this counterexample so that it is even a counterexample to
  the weaker Theorem \ref{thm:relbordismnonspin}.  Let $Y'$ be a connected
  sum of copies of $\overline{\bC\bP^2}$ with the same signature
  $-\frac{d}{3}(d^2-4)$ as $Y$, and let $X'=Y'\times [0,1]$.
  Then $X'$ obviously has a product metric of {\psc}, while
  $X$ does not; yet they represent the
  same class in $\Omega_5(B*, B(*\sqcup *);\bZ)\cong \Omega_4\cong \bZ$
  since that class is detected just by the signature.
\item Suppose $X$ is a compact manifold with boundary, of dimension
  $n\ge 6$ so that our theorems apply.  A major problem is to try
  to determine the homotopy type of $\cR^+(X)_{H>0}$ ({\psc} metrics on $X$
  with positive mean curvature on $\p X$) when this space
  is non-empty, or at least to give nontrivial
  lower bounds on the homotopy groups.  As we mentioned before,
  \cite[Corollary 34]{BH} shows that this space is homotopy equivalent
  to $\cR^+(M)^{\bZ/2}$, the reflection-invariant {\psc} metrics on
  the double $M=\Dbl(X,\p X)$.  And \cite[Corollary 40]{BH} shows that 
  when $\p X$ admits a {\psc} metric, then considering the
  space of {\psc} metrics with product boundary conditions would
  give us the same homotopy type.  In this case
  (when there is a {\psc} metric which is a product metric near
  the boundary) and when everything is spin,
  \cite{MR3681394,MR3956897} give lower bounds on the homotopy
  groups of $\cR^+(X)_{H>0}$ in terms of the $KO$-groups of the
  $C^*$-algebra of the fundamental group.  It is possible that a similar
  analysis, using APS methods as in \cite[\S2.1]{BH}, would also
  work with the weaker boundary condition $H>0$.

  The references just cited, and other similar ones, also say
  something about the homotopy
  type of $\cR^+(M)$ (the {\psc} metrics with no equi\-variance condition),
  but it's not immediately clear how this translates into information
  about $\cR^+(M)^{\bZ/2}$.  An example worked out in
  \cite[Corollary 45 and Remark 46]{BH} does give a case where
  $\cR^+(X)_{H>0}$ has infinitely many path components and infinite
  homotopy groups in all dimensions, but the construction is somewhat
  \emph{ad hoc}.
\end{enumerate}
\bibliographystyle{amsplain}

\bibliography{PSCDouble}

\end{document}